\documentclass[a4paper,12pt]{article}

\usepackage[textwidth=125mm, textheight=195mm]{geometry}
\usepackage[english]{babel}
\usepackage{graphicx}
\usepackage{amsmath}
\usepackage{amsfonts}
\usepackage{amsthm}
\usepackage{epstopdf}
\usepackage{color}

\begin{document}

\newcommand{\wk}{\mbox{$\,<$\hspace{-5pt}\footnotesize )$\,$}}

\numberwithin{equation}{section}
\newtheorem{teo}{Theorem}
\newtheorem{lemma}{Lemma}

\newtheorem{coro}{Corollary}
\newtheorem{prop}{Proposition}
\theoremstyle{definition}
\newtheorem{definition}{Definition}
\theoremstyle{remark}
\newtheorem{remark}{Remark}

\newtheorem{scho}{Scholium}
\newtheorem{open}{Question}
\newtheorem{example}{Example}
\numberwithin{example}{section}
\numberwithin{lemma}{section}
\numberwithin{prop}{section}
\numberwithin{teo}{section}
\numberwithin{definition}{section}
\numberwithin{coro}{section}
\numberwithin{figure}{section}
\numberwithin{remark}{section}
\numberwithin{scho}{section}

\bibliographystyle{abbrv}

\title{Convex analysis in normed spaces and metric projections onto convex bodies}
\date{}

\author{Vitor Balestro\footnote{Corresponding author}  \\ Instituto de Matem\'{a}tica e Estat\'{i}stica \\ Universidade Federal Fluminense \\ 24210201 Niter\'{o}i \\ Brazil \\ vitorbalestro@id.uff.br \and Horst Martini \\ Fakult\"{a}t f\"{u}r Mathematik \\ Technische Universit\"{a}t Chemnitz \\ 09107 Chemnitz\\ Germany \\ martini@mathematik.tu-chemnitz.de \and Ralph Teixeira \\ Instituto de Matem\'{a}tica e Estat\'{i}stica \\ Universidade Federal Fluminense \\ 24210201 Niter\'{o}i \\ Brazil \\ ralph@mat.uff.br}

\maketitle

\begin{abstract} We investigate metric projections and distance functions referring to  
convex bodies in finite-dimensional normed spaces. For this purpose we  
identify the vector space with its dual space by using, instead of  
the usual identification via the standard inner product, the Legendre  
transform associated with the given norm. This
approach yields re-interpretations of various properties of convex  
functions, and new relations between such functions and geometric  
properties of the studied norm are also derived.
\end{abstract}



\bigskip


\section{Introduction}

Let $(M,d)$ be a metric space, and let $C\subseteq M$ be a subset of $M$. The \emph{distance} of a given point $x \in M$ to $C$ is defined to be the number
\begin{align*} \mathrm{dist}(x,C) = \inf\{d(x,y):y \in C\}.
\end{align*}
If this number is finite, and if $y_0 \in C$ is such that $d(x,y_0) = \mathrm{dist}(x,C)$, then we say that $y_0$ is a \emph{metric projection} of $x$ onto $C$. 

Metric projections have been widely studied through the past  
decades and in a number of contexts (see, e.g., \cite{alimov}, \cite{asplund}, \cite{correa} and \cite{zili}, further references will be given throughout the text). A very special case is when one considers the metric projections onto a given \emph{convex body} (i.e., a compact, convex set with non-empty interior) $K$ in $\mathbb{R}^n$ endowed with the standard distance given for the Euclidean norm (which we denote by $|\cdot|$). For this case we know that, for example, \\

\noindent\textbf{(i)} $K$ is a \emph{Chebyshev set}, meaning that a metric projection onto $K$ exists and is unique for any $x \in \mathbb{R}^n$. Actually, every Chebyshev set of $\mathbb{R}^n$ is convex, see \cite{bunt} and \cite{deutsch}. For infinite-dimensional vector spaces this is more complicated, and we refer the reader to \cite{vlasov} and \cite{vlasov2},\\

\noindent\textbf{(ii)} the map $p_K\colon\mathbb{R}^n\rightarrow K$ which takes each $x \in \mathbb{R}^n$ to its metric projection onto $K$ is \emph{contracting}, meaning that $|p_K(x) - p_K(y)| \leq |x-y|$ for any $x,y \in \mathbb{R}^n$,\\

\noindent\textbf{(iii)} the distance function $\mathrm{dist}(\cdot,K)\colon\mathbb{R}^n\rightarrow\mathbb{R}$ is convex, and differentiable at $\mathbb{R}^n\setminus K$, even if $K$ is not \emph{smooth} (a convex body is said to be \emph{smooth} if it has a unique supporting hyperplane at each boundary point), \\

\noindent\textbf{(iv)} the gradient of the distance function can be described in terms of the outer normals of $K$.\\

We originally wanted to understand which of these (and further) properties remain true if we consider the metric given by an arbitrary norm in $\mathbb{R}^n$, instead of the usual norm given by the standard inner product. In this paper, we give precise answers to all the cases above except for \textbf{(ii)}, describing  
what properties the norm and/or the convex body have to satisfy such  
that the desired holds. For normed spaces, property \textbf{(ii)} was discussed in \cite{karlovitz}. We mention that regularity properties for distance functions have also been extensively investigated in various contexts (see, e.g, \cite{yanyan}, \cite{noll}, \cite{thibault} and \cite{zajicek2}). In the context of normed (or \emph{gauge}) spaces we mention the papers \cite{fitzpatrick}, \cite{giles}, \cite{giles2}, \cite{safdari} and \cite{zajicek}. However, to our best knowledge the following question was not explicitly answered: can we guarantee differentiability for the distance function to a convex body only assuming that the ambient norm is smooth and strictly convex (in the ``geometric" sense that its unit ball is smooth and strictly convex)? We answer this question positively, also relating the so-called \emph{norm gradient} of these functions to the Birkhoff orthogonality relation given by the norm. 

Seeing properties \textbf{(iii)} and \textbf{(iv)} in the context of normed  
spaces, we observed that the early theory of convex functions on $\mathbb{R}^n$ can be built \textbf{replacing the norm given by the standard inner product by a smooth and strictly convex norm}. Besides providing the complete understanding of differentiability of distance functions onto convex sets in normed spaces, we believe that this approach is interesting for itself. Using the Legendre transform of the norm, we define and study the notions of \emph{norm gradients} and \emph{norm sub-gradients} of convex functions, relating them also to the concept of Birkhoff orthogonality (to our best knowledge, this relation is also new).

We organize the text as follows: in Section \ref{convex} we recall some definitions and properties regarding the classical theory of convex functions, briefly discussing also the sub-linear functions. We begin to study metric projections onto convex sets in normed spaces in Section \ref{metric}, where we give necessary and sufficient conditions for the uniqueness (among other results). Section \ref{gradsubgrad} is devoted to introduce and investigate the so-called norm gradients and norm sub-gradients of convex functions. They are also used to ``detect" differentiability, and this is applied in Section \ref{distdiff} to study the differentiability of distance functions to convex bodies. 

To finish this introduction, we fix some basic concepts and notation. We say that a hyperplane $\mathbf{h} \subseteq \mathbb{R}^n$ \emph{supports} a given convex body $K$ at a point $x \in \partial K$ (where $\partial K$ denotes the  
\emph{boundary of} $K$) if $K$ is contained in one of the (closed) half-spaces determined when $\mathbf{h}$ is translated to pass through $x$. As it was previously mentioned, a convex body is \emph{smooth} if it is supported by a unique hyperplane at each boundary point. We also say that a convex body is \emph{strictly convex} if its boundary contains no line segment. Let $(X,||\cdot||)$ denote a \emph{normed or Minkowski space}, i.e., a finite dimensional, real Banach space. The \emph{unit ball} of $(X,||\cdot||)$ is the set $B := \{x \in X:||x|| \leq 1\}$, and it is immediate that $B$ is a convex body which is symmetric with respect to the origin. The boundary of the unit ball is the \emph{unit sphere} $\partial B = \{x \in X:||x|| = 1\}$. We say that a norm is \emph{smooth} (resp. \emph{strictly convex}) if its unit ball $B$ is a smooth convex body (resp. strictly convex body). Equivalently, a norm is strictly convex if and only if the triangle inequality is strict for linearly independent vectors (see \cite{martini1}). 

A norm $||\cdot||$ on a vector space $X$ induces an orthogonality relation which is known as \emph{Birkhoff orthogonality}. We say that a vector $x \in X$ is \emph{Birkhoff left-orthogonal} to a vector $y \in X$ if $||x+ty|| \geq ||x||$ for every $t \in \mathbb{R}$. We denote this relation by $x \dashv_B y$, and in this case we also say that $y$ is \emph{Birkhoff right-orthogonal} to $x$. It is easy to see that Birkhoff orthogonality is a homogeneous relation (meaning it is a relation between directions rather than vectors). Extending this concept, we say that a vector $x \in X$ is \emph{Birkhoff left-orthogonal to a hyperplane} $\mathbf{h}$ if $x\dashv_B z$ for every $z \in \mathbf{h}$. In this case, we also say that $\mathbf{h}$ is Birkhoff right-orthogonal to $x$, and this is similarly denoted by $x \dashv_B \mathbf{h}$. Birkhoff orthogonality between vectors and hyperplanes is related to the geometry of the unit ball: the norm is smooth if and only if each non-zero vector $x \in X$ admits a unique Birkhoff right-orthogonal hyperplane, and the norm is strictly convex if and only if each hyperplane $\mathbf{h}\subseteq X$ has a unique Birkhoff left-orthogonal direction. We refer the reader to \cite{alonso} for more information on Birkhoff orthogonality (and other orthogonality types in normed spaces).

\section{Convex functions}\label{convex}

A function $f\colon\mathbb{R}^n\rightarrow\mathbb{R}$ is said to be \emph{convex} if
\begin{align*} f(\lambda x + (1-\lambda)y) \leq \lambda f(x) + (1-\lambda)f(y)
\end{align*}
for any $x,y \in \mathbb{R}^n$ and $\lambda \in [0,1]$. Equivalently, a function $f\colon\mathbb{R}^n\rightarrow\mathbb{R}$ is convex if and only if its \emph{epigraph}
\begin{align*} \mathrm{epi}(f) = \{(x,c)\in\mathbb{R}^{n+1}:f(x) \leq c\}
\end{align*}
is a convex set. Replacing the domain by convex open subsets of ${\mathbb R}^n$, the definition can be extended, as well as the most part of our results. For the sake of simplicity, we work with functions defined over $\mathbb{R}^n$ and which only take values on $\mathbb{R}$ (and not on the extended line $\mathbb{R}\cup\{-\infty,+\infty\}$, as it is usual in the convex analysis literature).

We deal with two differentiability notions of convex functions. First, if $f\colon\mathbb{R}^n\rightarrow\mathbb{R}$ is a convex function, then we denote by $f'_+$ the \emph{one-sided directional derivative}
\begin{align*} f'_+(x,v) := \lim_{t\rightarrow 0^+}\frac{f(x+tv)-f(x)}{t},
\end{align*}
for any $x,v \in \mathbb{R}^n$. In our context (that of convex functions whose domain is $\mathbb{R}^n$), this limit always exists. Also, it is well-known that $f'_+(x,\cdot)$ is a sublinear map for each $x \in \mathbb{R}^n$. 
For proofs we refer the reader to \cite[Chapter 1]{schneider}. Letting $t \rightarrow 0^-$ in the limit above, we get the \emph{left-sided derivative} $f'_-(x,v)$. It is not difficult to see that $f'_-(x,v) = -f'_+(x,-v)$, and for that reason we work mainly with $f'_+$.

A function $f\colon U\rightarrow\mathbb{R}$, where $U\subseteq\mathbb{R}^n$ is an open set, is said to be (\emph{Fr\'{e}chet}) \emph{differentiable} at $x \in U$ if there exists a linear map $df_x\colon\mathbb{R}^n\rightarrow\mathbb{R}$ such that
\begin{align*} \lim_{||v||\rightarrow 0}\frac{|f(x+v)-f(x)-df_x(v)|}{||v||} = 0,
\end{align*}
where the limit is taken in the metric given by the norm. It is immediate to notice that this definition coincides with the differentiability definition given by the Euclidean norm. For convex functions, there are several characterizations of Fr\'{e}chet differentiability, and next we state the ones which are important for us.

\begin{prop}\label{diffconvexchar} Let $f\colon\mathbb{R}^n\rightarrow\mathbb{R}$ be a convex function, and let $x \in \mathbb{R}^n$. For simplicity, we write $c := f(x)$. The  following statements are equivalent:\\

\noindent\textbf{\emph{(i)}} $f$ is Frech\'{e}t differentiable at $x$,\\

\noindent\textbf{\emph{(ii)}} the one-sided derivative $f'_+(x,\cdot)$ is a linear map, and\\

\noindent\textbf{\emph{(iii)}} $\mathrm{epi}(f)$ is supported by a unique hyperplane at $(x,c)$.\\

\noindent In this case, we have that $f'_+(x,\cdot) = df_x(\cdot)$.
\end{prop}
\begin{proof}
Again, for the proofs we refer to \cite[Chapter 1]{schneider}.

\end{proof}

In the points where $f$ is not differentiable, we still can guarantee that the one-sided derivative $f'_+(x,\cdot)\colon\mathbb{R}^n\rightarrow\mathbb{R}$ is \emph{sub-linear}, meaning that it has the following properties: \\

\noindent\textbf{i. (sub-additivity)} $f'_+(x,u+v) \leq f'_+(x,u)+f'_+(x,v)$ for every $u,v \in\mathbb{R}^n$, and\\

\noindent\textbf{ii. (positive homogeneity)} $f'_+(x,\lambda u) = \lambda f'_+(x,u)$ for any $u \in \mathbb{R}^n$ and every $\lambda \geq 0$. \\

The reader may carefully notice that sub-linear functions are, in particular, convex. Hence, if $g\colon\mathbb{R}^n\rightarrow\mathbb{R}$ is a sub-linear function, then it has well-defined one-sided derivatives $g'_+(x,u)$ for any $x,u \in \mathbb{R}$, and the maps $g'_+(x,\cdot)$ are also sub-linear. A \emph{linearity direction} of a sub-linear function $g$ is a vector $u \in \mathbb{R}^n$ for which $g(u) = -g(-u)$. The set of linearity directions of $g$ is denoted by $\mathrm{lin}(g)$. 

The next lemma establishes some important properties of sub-linear functions that we will use later. The proof is given in \cite[Chapter 1]{schneider}.

\begin{lemma}\label{sublinearproperties} Let $g\colon\mathbb{R}^n\rightarrow\mathbb{R}$ be a sub-linear function. The following statements hold:\\

\noindent\textbf{\emph{(a)}} $g'_+(x,\lambda x) = \lambda g(x)$ for any $x \in \mathbb{R}^n$ and $\lambda \in \mathbb{R}$, \\

\noindent\textbf{\emph{(b)}} $g'_+(x,u) \leq g(u)$ for any $x,u \in \mathbb{R}^n$,\\

\noindent\textbf{\emph{(c)}} $\mathrm{lin}(g)$ is a vector subspace of $\mathbb{R}^n$, \\

\noindent\textbf{\emph{(d)}} $\mathrm{lin}(g'_+(x,\cdot)) \supseteq \mathrm{lin}(g) + \mathrm{span}(x)$ for every $x \in \mathbb{R}^n$,  \\

\noindent\textbf{\emph{(e)}} $\mathrm{lin}(g) = \mathbb{R}^n$ if and only if $g$ is linear. 
\end{lemma}

If $f$ is a convex function, then it is easy to notice that for each $c\in\mathbb{R}$ the \emph{sub-level set} $\{f \leq c\} := \{z \in \mathbb{R}^n:f(z) \leq c\}$ is convex. Despite being simple, the next proposition provides important information about the structure of the sub-level sets of convex functions (namely, a point which is not a global minimizer lies in the boundary of some sub-level set).

\begin{prop}\label{convexsublevel} Let $f\colon\mathbb{R}^n\rightarrow\mathbb{R}$ be a convex function and assume that $f(x) = c$. Then precisely one of the following statements holds:\\

\noindent\textbf{\emph{(i)}} $\{f < c\} \neq \emptyset$ and $x \in \partial \{f \leq c\}$, \\

\noindent\textbf{\emph{(ii)}} $c$ is the global minimum of $f$.
\end{prop}
\begin{proof} If $c$ is not the minimum of $f$, then there exists a point $z \in \mathbb{R}^n$ such that $f(z) < c = f(x)$. To prove that $x \in \partial \{f\leq c\}$, let $\varepsilon > 0$ and put $w = x-z$. Then
\begin{align*} f(x-\varepsilon w) = f((1-\varepsilon)x + \varepsilon z) \leq (1-\varepsilon)f(x) + \varepsilon f(z) < c,
\end{align*}
from where we get that $x-\varepsilon w \in B(x,\varepsilon)\cap\{f < c\}$. On the other hand, writing
\begin{align*} x = \frac{1}{1+\varepsilon}(x+\varepsilon w) + \frac{\varepsilon}{1+\varepsilon}z\,,
\end{align*}
we get
\begin{align*} c = f(x) \leq \frac{1}{1+\varepsilon}f(x+\varepsilon w) + \frac{\varepsilon}{1+\varepsilon}f(z) < \frac{1}{1+\varepsilon}f(x+\varepsilon w) + \frac{\varepsilon}{1+\varepsilon}c,
\end{align*}
and this leads to $f(x+\varepsilon w) > c$. Hence $x + \varepsilon w \in B(x,\varepsilon)\cap\{f > c\}$. Since this holds for any $\varepsilon > 0$, it follows that any neighborhood of $x$ contains points which lie inside and outside of $\{f\leq c\}$. Hence $x \in \partial \{f\leq c\}$.

On the other hand, it is clear that if $c$ is the global minimum of $f$, then $\{f < c\} = \emptyset$. As a remark, notice that by continuity we have that $\{f < c\}$ is an open (convex) set. Moreover, observe that if $c$ is not a global minimum, then $\mathrm{int}\{f \leq c\} = \{f < c\}$ and $\{f = c\} = \partial \{f < c\}$.

\end{proof}

\begin{prop} Let $f\colon\mathbb{R}^n\rightarrow\mathbb{R}$ be a convex function, and assume that $f$ is differentiable at $x \in \mathbb{R}^n$. If $f(x) = c$ is not a global minimum of $f$, then the sub-level set $\{f\leq c\}$ is supported by a unique hyperplane at $x$.
\end{prop}
\begin{proof} Let $\mathbf{h}$ be a supporting hyperplane of $f$ at $x$, and let $z \in \mathbf{h}$. Observe that $f(x+\varepsilon z) \geq f(x)$ for \textbf{every} $\varepsilon \in \mathbb{R}$, meaning that $f$ grows in both directions $z$ and $-z$. That is,
\begin{align*} f'_+(x,\pm z) \geq 0.
\end{align*}
On the other hand, since $f$ is differentiable at $x$, we have
\begin{align*} 0 \leq f'_+(x,z) = df_x(z) = f'_-(x,z) = -f'_+(x,-z) \leq 0,
\end{align*}
from which it follows that $df_x(z) = 0$. This shows that $df_x$ vanishes in any direction which supports $\{f \leq c\}$ at $x$. Consequently, if $\{f \leq c\}$ has more than one supporting hyperplane at $x$, then $df_x = 0$, and this contradicts the fact that $f(x) = c$ is not a global minimum. 

\end{proof}

\section{The metric projection}\label{metric}

Roughly speaking, the \emph{metric projection} of a point $x$ onto a convex body $K$ is the point of $\partial K$ where the minimum distance from $x$ to points from $K$ is attained. For the Euclidean case, this concept is studied in \cite[Chapter 1]{schneider}, and in this section we extend the ideas presented there to Minkowski spaces. If $(X,||\cdot||)$ is a normed space, then the \emph{distance} from a point $x \in X$ to a set $C\subseteq X$ is defined as
\begin{align*} \mathrm{dist}(x,C) = \inf\{||x-y||:y \in C\}.
\end{align*}
We are mainly interested in studying metric projections of points onto convex bodies. The reader must be aware of the fact that, in general, strict convexity and smoothness are not assumed, neither for the norm nor for the considered convex body.

\begin{teo}\label{metricprojectionteo1} Let $(X,||\cdot||)$ be a finite-dimensional vector space, and let $K\subseteq X$ be a convex body. For each point $x \in X$, there exists a point $p_K(x) \in K$ such that
\begin{align*} \mathrm{dist}(x,K) = ||x - p_K(x)|| \leq ||x-y||
\end{align*}
for any $y \in K$. Moreover, if $x \notin K$, then $p_K(x) \in \partial K$, and the vector $x - p_K(x)$ is (left) Birkhoff orthogonal to some supporting hyperplane of $K$ at $p_K(x)$.
\end{teo}
\begin{remark} The point $p_K(x)$ is called a \emph{metric projection} of $x$ onto $K$. Of course, if $x \in K$, then $p_K(x) = x$.
\end{remark}
\begin{proof} The existence of $p_K(x)$ comes straightforwardly from the compactness of $K$ and from the continuity of the norm. It is also clear that if $x \notin K$, then $p_K(x) \in \partial K$. Indeed, if $x \notin K$ and $y_0 \in \mathrm{int}K$, then the segment connecting $x$ to $y_0$ cuts the boundary of $K$ (at a point $z_0$, say) with $||x - y_0|| > ||x - z_0||$.

For the other claim, assume that $x \notin K$ and write $\rho:=\mathrm{dist}(x,K) = ||x-p_K(x)||$. Then $K$ and $B(x,\rho)$ are convex bodies which clearly intersect along their boundaries, that is,
\begin{align*} K\cap B(x,\rho) \subseteq \partial K\cap \partial B(x,\rho).
\end{align*}
In fact, if this does not hold, then we would have $\mathrm{dist}(x,K) < \rho$. It follows that $\mathrm{int}(K)\cap\mathrm{int}(B(x,\rho)) = \emptyset$, and hence $\mathrm{int}(K)$ and $\mathrm{int}(B(x,\rho))$ can be (properly) separated by some hyperplane $\mathbf{h}$ (see \cite[Theorem 1.3.8]{schneider}). Since we clearly have that $p_K(x) \in \mathbf{h}$, we get that $\mathbf{h}$ supports both $B(x,\rho)$ and $K$ at $p_K(x)$. The fact that $\mathbf{h}$ supports $B(x,\rho)$ at $p_K(x)$ gives that $x-p_K(x)$ is Birkhoff left-orthogonal to $\mathbf{h}$ (translated to pass through the origin).

\end{proof}

From now on, we sometimes will denote the distance function $\mathrm{dist}(\cdot,C)\colon X\rightarrow\mathbb{R}$ to a given non-empty compact set $C$ by $d_C(\cdot)$. In general, we work with convex bodies, but some results are still true if we only demand compactness.

\begin{prop}\label{distanceconvex} Let $K\subseteq X$ be a convex body. The distance function $d_K\colon X\rightarrow\mathbb{R}$ is a convex function.
\end{prop}
\begin{proof} Let $x,y \in X$ and $\lambda \in (0,1)$. Since $K$ is convex, the segment connecting $p_K(x)$ and $p_K(y)$ is contained in $K$, and hence
\begin{align*} d_K((1-\lambda)x+\lambda y) \leq ||(1-\lambda)x+\lambda y - ((1-\lambda)p_K(x)+\lambda p_K(y))|| \leq \\ \leq (1-\lambda)||x-p_K(x)|| + \lambda ||y-p_K(y)|| = (1-\lambda)d_K(x) + \lambda d_K(y),
\end{align*}
proving that $d_K$ is convex.

\end{proof}

\begin{lemma} Let $C\subseteq X$ be a non-empty compact set. The distance function $d_C\colon X\rightarrow\mathbb{R}$ is a weak contraction, that is, we have
\begin{align*} |d_C(x) - d_C(y)| \leq ||x-y||
\end{align*}
for any $x,y \in X$.
\end{lemma}
\begin{proof} Without loss of generality, assume that $d_C(y) \geq d_C(x)$, and write $y = x + z$. Let $p_C(x)$ and $p_C(y)$ be metric projections of $x$ and $y$ onto $C$, respectively. Then we have
\begin{align*} |d_C(y) - d_C(x)| = d_C(y) - d_C(x) = ||y - p_C(y)|| - ||x-p_C(x)|| = \\= ||x+z - p_C(x+z)|| - ||x-p_C(x)|| \leq ||x+z - p_C(x)||-||x-p_C(x)|| \leq \\ \leq ||z|| = ||x-y||,
\end{align*}
and this concludes the proof.

\end{proof}

Next we discuss uniqueness of the metric projection of a given point. As it is explained now, we have no uniqueness only in the case that $B_X$ is not strictly convex and $\partial K$ contains a line segment which is parallel to some line segment of $\partial B_X$. The metric projection onto $K$ is said to be \emph{well-defined} if it is unique for each $x \in X$.

\begin{prop} If $(X,||\cdot||)$ is a normed space and $K\subseteq X$ is a convex body, then a metric projection $p_K(x)$ of an exterior point $x \notin K$ is not unique if and only if $p_K(x)$ is contained in some non-degenerate line segment of $\partial K$ which is parallel to some non-degenerate line segment in the boundary $\partial B_X$ of the unit ball.
\end{prop}

\begin{proof} Assume that there are two distinct points $y_0,z_0 \in \partial K$ where the distance from $x$ to $K$ is attained, that is, with
\begin{align*} ||x-y_0|| = ||x-z_0|| = \mathrm{dist}(x,K) :=\rho.
\end{align*}
Hence $y_0,z_0 \in B(x,\rho)\cap K$. Since $B(x,\rho)\cap K \subseteq \partial K\cap \partial B(x,\rho)$, we have two distinct points in the intersection of the boundaries of two convex bodies. This is only possible if both boundaries contain the segment $\mathrm{seg}[y_0,z_0]$ connecting these points.

It follows that $\mathrm{seg}[y_0,z_0]$ is a non-degenerate line segment of $\partial K$. Since $\mathrm{seg}[y_0,z_0]$ is also a line segment in the boundary of $B(x,\rho)$, it corresponds to a (parallel) line segment in $\partial B_X$.

\end{proof}

\begin{coro} If $(X,||\cdot||)$ is a strictly convex normed space, then the metric projection on a given convex body is unique for any point $x \in X$.
\end{coro}

Let $z\in \partial K$. An \emph{outer normal} of $K$ at $z$ is any outward pointing unit vector which is Birkhoff left-orthogonal to some supporting hyperplane of $K$ at $z$. 

\begin{prop}\label{outernormal} Let $x \notin K$, and let $p_K(x)$ be a metric projection of $x$ onto $K$. Then the outward pointing unit vector
\begin{align*} \eta_K(x) := \frac{x-p_K(x)}{\mathrm{dist}(x,K)}
\end{align*}
is an outer normal of $K$ at $p_K(x)$.
\end{prop}
\begin{proof} The fact that $\eta_K(x)$ is Birkhoff left-orthogonal to some supporting hyperplane of $K$ at $p_K(x)$ comes from Theorem \ref{metricprojectionteo1}. Since $\eta_K(x)$ is clearly unit and outward pointing, we get that $\eta_K(x)$ is an outer normal of $K$.

\end{proof}

In what follows, we denote the line passing through $z$ in the direction of $u$ as $z + \mathbb{R}u$. The ray starting at $z$ with the direction of $u$ will be denoted by $z + \mathbb{R}^+u$. Next we investigate the ``converse" of the idea that the (normalized) vector subtraction between an exterior point and its metric projection is an outer normal.

\begin{prop}\label{propsun} Let $K$ be any convex body in $X$, and let $x \in X\setminus K$. Then $p_K(x)$ is a metric projection for any $y \in p_K(x) +\mathbb{R}^+\eta_K(x)$.
\end{prop}
\begin{proof} We write $\rho:=\mathrm{dist}(x,K)$. By Theorem \ref{metricprojectionteo1}, let $\mathbf{h}$ be a hyperplane which supports both $K$ and $B(x,\rho)$ at $p_K(x)$. If $y \in p_K(x) + \mathbb{R}^+\eta_K(x)$, then $y - p_K(x)$ has the  direction of $x-p_K(x)$, and hence it is (left) Birkhoff orthogonal to $\mathbf{h}$. Writing $\alpha := ||y-p_K(x)||$, it follows that $\mathbf{h}$ supports $B(y,\alpha)$ at $p_K(x)$. Since $B(y,\alpha)$ clearly lies in the same half-space determined by $\mathbf{h}$ as $B(x,\rho)$, we get that $K\cap\mathrm{int}(B(y,\alpha)) = \emptyset$. Consequently, we have that $||y - z|| \geq \alpha$ for all $z \in K$, and hence $p_K(x)$ is a metric projection of $y$ onto $K$ (because $||y-p_K(x)|| = \alpha$).

\end{proof}

\begin{remark} A set with the property proved in this proposition is often called a \emph{sun} (see, e.g., \cite{brown,hetzelt,narang} and references therein).
\end{remark}

\begin{scho}If $z \in \partial K$, then any point in the ray starting at $z$ and going into the direction of an outer normal of $K$ at $z$ has $z$ as a metric projection.
\end{scho}
\begin{proof} Let $z \in \partial K$ and assume that $u$ is an outer normal of $K$ at $z$. Let $\mathbf{h}$ be a supporting hyperplane of $K$ at $z$ which is (right) Birkhoff orthogonal to $u$. If $w = z + \beta u$ for some $\beta > 0$, then $z \in \partial B(w,\beta)$ and $w - z$ has the direction of $u$, from where $\mathbf{h}$ supports $B(w,\beta)$ at $z$. Clearly, $K$ and $B(w,\beta)$ lie in distinct half-spaces determined by $\mathbf{h}$, from where we get that $K\cap \mathrm{int}(B(w,\rho)) = \emptyset$. It follows that $||w - y|| \geq \beta$ for every $y \in K$, and this shows that $z$ is a metric projection of $w$ onto $K$, because $||w-z|| = \beta$.

\end{proof}

We have a geometric consequence that will be useful to understand the regularity of the distance function. A \emph{parallel set} (or \emph{parallel body}) of a given convex body $K$ is a set of the type
\begin{align*} K+\delta B := \{x+\delta y: x \in K, \ y \in B\} = \{x \in X:\mathrm{dist}(x,K)\leq \delta\},
\end{align*}
for $\delta > 0$. The equality above is immediate, and we will skip the proof. Also, it is easy to see that each parallel set $K+\delta B$ is a convex body.

If we assume that $B$ is smooth (in the sense that it has a unique supporting hyperplane at each boundary point), then $K+\delta B$ is smooth for any $\delta > 0$, even if $K$ is not smooth (see, e.g., \cite{krantz}).

\begin{prop}\label{birkhoffparallelset} Let $K$ be a convex body, and let $z \in \partial K$ be a boundary point. Let $u_K(z)$ be an outer normal vector of $K$ at $z$, and let $\delta > 0$. If $x = z + \delta u_K(z)$, then $u_K(z)$ is a Birkhoff normal vector of the parallel body $K+\delta B$ at $x$.
\end{prop}
\begin{proof} Denote by $\mathbf{h}$ the hyperplane which is Birkhoff right-orthogonal to the vector $u_K(z)$. It suffices to prove that $\mathbf{h}$ supports $K+\delta B$ at $x$.  Abusing a little of the notation, we assume that $\mathbf{h}$ is translated to pass through $x$. Denoting the (closed) half-spaces regarding $\mathbf{h}$ by $\mathbf{h}^+$ and $\mathbf{h}^-$, we may assume, without loss of generality, that $K\subseteq \mathbf{h}^-$. Hence we have that $K+\delta B\subseteq (\mathbf{h}+\delta u_K(z))^-$. Indeed, if $y \in \mathrm{int}(\mathbf{h}+\delta u_K(z))^+$, then $||y - w|| > \delta$ for any $w \in K$, and this, together with the compactness of $K$, leads to $\mathrm{dist}(y,K) > \delta$, from where we get that $y \notin K+\delta B$. It follows that $\mathbf{h}+\delta u_K(z)$ supports $K+\delta B$ at $x$, yielding that $u_K(z)$ is a Birkhoff normal vector of $K+\delta B$ at $x$.

\end{proof}

We know that the distance function to a convex body is convex, and hence it is continuous. However, we said nothing about the continuity of the metric projection (in case it is unique) thus far. This will be clarified next. In the last section of the paper, we will prove that the continuity of the metric projection gives that the derivative of the distance function is continuous (that is, $d_K$ is of class $C^1$).

\begin{prop}\label{metricprojectioncontinuous} Let $K\subseteq X$ be a convex body, and assume that the metric projection $p_K\colon X\rightarrow K$ is well-defined. Then $p_K$ is continuous. 
\end{prop} 
\begin{proof} It is clear that $p_K$ is continuous in the interior of $K$. Hence we consider the case where $x \in X$ is not an interior point of $K$. Let $(x_n)_{n\in\mathbb{N}}$ be a sequence converging to $x$, and assume that $p_K(x_n)$ does not converge to $p_K(x)$. Since $K$ is compact, it follows that there exists a subsequence $p_K(x_{n_j})$ converging to a point $w \in K$ with $w \neq p_K(x)$. From the continuity of the norm and of the distance function we get
\begin{align*} ||x-w|| = \lim_{j\rightarrow\infty}||x_{n_j}-p_K(x_{n_j})|| = \lim_{j\rightarrow\infty}\mathrm{dist}(x_{n_j},K) = \mathrm{dist}(x,K),
\end{align*}
from where $w$ is a metric projection of $x$ onto $K$. This contradicts the uniqueness of the metric projection. Hence $p_K(x_n)\rightarrow p_K(x)$, and this shows that $p_K$ is continuous at $x$. 

\end{proof}

\section{Norm gradients, and norm sub-gradients}\label{gradsubgrad}

For simplicity, we work from now on in the vector space $\mathbb{R}^n$; so we can use standard notions of differentiability. Also, we \textbf{always} suppose that the norm is smooth and strictly convex; so the Birkhoff orthogonality is ``well-behaving". Let $f\colon U\rightarrow\mathbb{R}$ be a (Fr\'{e}chet) differentiable function, where $U\subseteq\mathbb{R}^n$ is an open set. A number $c \in \mathrm{im}(f)$ is said to be a \emph{regular value} of $f$ if $df_x$ is surjective for each $x \in f^{-1}(c)$. It is well known that the pre-image of any regular value is a hypersurface of $\mathbb{R}^n$. The \emph{gradient vector} $\mathrm{grad}f(x)$ of $f$ at $x \in U$ is the (unique) vector such that $df_x(v) = \langle \mathrm{grad} f(x),v\rangle$ for each $v\in\mathbb{R}^n$. We say that $f$ is a \emph{function of class $C^1$} if $\mathrm{grad}f(x)$ is continuous. This definition relies on an inner product fixed in $\mathbb{R}^n$, and we would like to have a definition consistent with the geometry given by a (smooth and strictly convex) arbitrary norm.

\begin{definition} Let $U\subseteq\mathbb{R}^n$ be an open set, and let $f\colon U\subseteq\mathbb{R}^n\rightarrow\mathbb{R}$ be a differentiable function. Let $c$ be a regular value of $f$. The \emph{norm gradient} of $f$ at a point $x \in f^{-1}(c)$ is the unique vector $\nabla f(x) \in \mathbb{R}^n$ with the following properties:\\

\noindent\textbf{(i)} $\nabla f(x) \dashv_B T_x(f^{-1}(c))$, and \\

\noindent\textbf{(ii)} $df_x(\nabla f(x)) = ||\nabla f(x)||^2$.\\

\noindent If $c \in \mathrm{im}(f)$ is not a regular value and $x \in f^{-1}(c)$, then we simply put $\nabla f(x) = 0$.
\end{definition}

\begin{remark} From now on we always use the symbol $\nabla$ as a notation for the norm gradient, despite the fact that this notation is often used for the standard Euclidean gradient. This choice is justified because we will not work with the Euclidean gradient throughout the text.
\end{remark}

It is not difficult to see that the norm gradient is well-defined. If $c$ is a regular value, then $f^{-1}(c)$ is a hypersurface, and hence it has an $(n-1)$-dimensional tangent hyperplane $T_x(f^{-1}(c))$ at any point $x \in f^{-1}(c)$. This hyperplane supports the unit ball of the norm at a point which determines the Birkhoff left-orthogonal direction to $T_x(f^{-1}(c))$. The condition \textbf{(ii)} gives both, the choice of an orientation and a normalization.

This definition makes sense for $C^1$ functions, since in this case we can define the tangent space of the pre-image of a regular value. However, we can relax this regularity hypothesis when dealing with convex functions, and we are mostly concerned with this case. Based on Proposition \ref{convexsublevel}, we define the \emph{norm gradient} of a convex function $f\colon\mathbb{R}^n\rightarrow\mathbb{R}$ at a point $x \in \mathbb{R}^n$ where $f$ is differentiable. If $x \in \mathbb{R}^n$ is such that $c = f(x)$ is not a global minimum, then we take the unique non-zero vector $\nabla f(x)$ such that\\

\noindent\textbf{(a)} $\nabla f(x)$ is Birkhoff left-orthogonal to the (unique) supporting hyperplane of $\{f\leq c\}$ at $x$, \\

\noindent\textbf{(b)} $df_x(\nabla f(x)) = ||\nabla f(x)||^2$.\\

\noindent Otherwise (that is, if $c$ is a global minimum), we put $\nabla f(x) = 0$.

Denote by $(\mathbb{R}^n)^*$ the dual space of $\mathbb{R}^n$, and assume that $||\cdot||$ is a smooth and strictly convex norm on $\mathbb{R}^n$. The \emph{Legendre transform} on $(\mathbb{R}^n,||\cdot||)$ is the map $\mathcal{L}\colon\mathbb{R}^n\rightarrow(\mathbb{R}^n)^*$ which associates to each non-zero vector $x \in \mathbb{R}^n$ the unique linear functional $\mathcal{L}(x) \in (\mathbb{R}^n)^*$ such that\\

\noindent\textbf{(a)} $\mathrm{ker}\mathcal{L}(x)$ is the (unique) hyperplane which is Birkhoff right-orthogonal to $x$, \\

\noindent\textbf{(b)} $\mathcal{L}(x)\cdot x = ||x||^2$.\\

\noindent We also define $\mathcal{L}(0) = 0$, that is, the null functional in $(\mathbb{R}^n)^*$. It is immediate to observe that $\mathcal{L}$ is a bijection. 

\begin{remark} Notice very carefully that, for defining the Legendre  
transform, we do not need the norm to be strictly convex. However, we always assume that this hypothesis is true, because we are more interested in this case. 
\end{remark}

\begin{remark} If $|\cdot|$ is the Euclidean norm in $\mathbb{R}^n$ derived from the standard inner product $\langle\cdot,\cdot\rangle$, then it is easy to see that the Legendre transform of $|\cdot|$ is given as $\mathcal{L}(x) = \langle x,\cdot\rangle$. That is, in this case the Legendre transform is the natural isomorphism between $\mathbb{R}^n$ and $(\mathbb{R}^n)^*$ given by the standard inner product. 
\end{remark}

\begin{lemma} The Legendre transform is homogeneous of degree one, that is, $\mathcal{L}(\alpha x) = \alpha\mathcal{L}(x)$ for any $\alpha \in \mathbb{R}$ and $x \in \mathbb{R}^n$. 
\end{lemma}
\begin{proof} Of course, we may assume that $\alpha \neq 0$ and $x \neq 0$, since otherwise the proof is trivial. Let $\mathbf{h}$ be the hyperplane such that $x \dashv_B\mathbf{h}$. Then we also have that $\alpha x \dashv_B \mathbf{h}$, meaning that the linear functionals $\mathcal{L}(\alpha x)$ and $\alpha\mathcal{L}(x)$ have the same kernel $\mathbf{h}$. Now we observe that
\begin{align*} \mathcal{L}(\alpha x)\cdot(\alpha x) = ||\alpha x||^2 = \alpha^2||x||^2 = \alpha\mathcal{L}(x)\cdot(\alpha x),
\end{align*}  
that is, the linear functionals $\mathcal{L}(\alpha x)$ and $\alpha\mathcal{L}(x)$ take the same value at $\alpha x$. This concludes the proof, since now we have that both linear functionals agree on a basis of $\mathbb{R}^n$. 

\end{proof}

We also have that the Legendre transform is a norm-preserving map when one considers the \emph{dual norm} on $(\mathbb{R}^n)^*$. We prove this next. 

\begin{prop} Let $||\cdot||_*$ be the \emph{dual norm} on $(\mathbb{R}^n)^*$, defined as
\begin{align*} ||\phi||_* := \sup\{\phi(x):x \in B\}. 
\end{align*}
The Legendre transform is a norm-preserving map of $(\mathbb{R}^n,||\cdot||)$ onto $((\mathbb{R}^n)^*,||\cdot||_*)$.
\end{prop}
\begin{proof} From the homogeneity, it suffices to prove that $||\mathcal{L}(x)||_* = 1$ whenever $||x|| = 1$. Let $x$ be a unit vector, and assume that $\mathbf{h}$ is the hyperplane such that $x \dashv_B \mathbf{h}$. It is easy to see from the definition that $\mathcal{L}(x)\cdot y$ is positive if and only if $y$ lies in the open half-space determined by $\mathbf{h}$ which contains $x$ ($\mathbf{h}^+$, say). Also, $\mathbf{h}$ supports the unit ball at $x$, and consequently any vector $y \in B\cap\mathbf{h}^+$ can he written as
\begin{align*} y= \frac{x+z}{||x+z||}
\end{align*}
for some $z \in \mathbf{h}$. Since $||x+z|| \geq ||x||$, we get
\begin{align*} \mathcal{L}(x)\cdot y = \mathcal{L}(x)\cdot\frac{x+z}{||x+z||} = \frac{\mathcal{L}(x)\cdot x}{||x+z||} = \frac{||x||^2}{||x+z||} \leq ||x||^2  = 1,
\end{align*}
and hence $||\mathcal{L}(x)||_* = \sup\{\mathcal{L}(x)\cdot v:v\in B\} \leq 1$ is obtained. On the other hand, we have that $x \in B$ and $\mathcal{L}(x)\cdot x = ||x||^2 = 1$, and this yields $||\mathcal{L}(x)||_* = 1 = ||x||$. 

\end{proof}

Our definition of the Legendre transform may be a little bit intrincated to work with, but it has the advantage of not demanding any differentiability properties of the norm (only the geometric features of being smooth and strictly convex). Actually, later our approach will give an easy proof of the fact that norms with these geometric properties must be of type $C^1$. Even without assuming \emph{a priori} differentiability properties, we can guarantee that the Legendre transform is continuous. 

\begin{prop} Let $||\cdot||$ be a smooth and strictly convex norm on $\mathbb{R}^n$. Then the associated Legendre transform $\mathcal{L}\colon(\mathbb{R}^n,||\cdot||)\rightarrow((\mathbb{R}^n)^*,||\cdot||_*)$ is continuous.
\end{prop}
\begin{proof} First, observe that continuity at the origin comes immediately from the fact that the Legendre transform is norm-preserving. Also, since the Legendre transform is homogenous of degree $1$ it suffices to show that for any sequence $(x_n)_{n\in\mathbb{N}}$ of unit vectors converging to a (unit) vector $x$ we have that $\mathcal{L}(x_n)\rightarrow\mathcal{L}(x)$ in the dual norm as $n\rightarrow\infty$. We will divide the proof in steps.\\

\noindent\emph{First step.} We prove that for each fixed $v \in \mathbb{R}^n$ we have convergence $\mathcal{L}(x_n)\cdot v\rightarrow\mathcal{L}(x)\cdot v$ as $n\rightarrow \infty$ (pointwise convergence). Denote by $\mathbf{h}_n$ the supporting hyperplane of $B$ at $x_n$, and by $\mathbf{h}$ the supporting hyperplane of $B$ at $x$. For each $n\in\mathbb{N}$, we can write
\begin{align*} v = \alpha_nx_n + z_n
\end{align*}
for some $\alpha_n \in \mathbb{R}$ and $z_n \in \mathbf{h}_n$. Similarly, we write $v = \alpha x + z$, with $z \in \mathbf{h}$. We claim that, as $n\rightarrow \infty$, we have $\alpha_n\rightarrow \alpha$ (and, as a consequence, $z_n\rightarrow z$). First, observe that
\begin{align*} ||v|| = ||\alpha_nx_n+z_n|| = |\alpha_n|\cdot\left|\left|x_n+\frac{z_n}{\alpha_n}\right|\right| \geq |\alpha_n|\cdot||x_n|| = |\alpha_n|
\end{align*}
whenever $\alpha_n \neq 0$, where the last inequality comes from $x_n\dashv_Bz_n$. This shows that $(\alpha_n)$ is bounded. Suppose that some subsequence $\alpha_{n_k}\rightarrow\beta \neq \alpha$. Then $z_{n_k}\rightarrow v-\beta x$, and since Birkhoff orthogonality is a continuous relation, it follows that $x\dashv_B (v-\beta x)$. Therefore, we have two distinct decompositions of $v$ in the sum $\mathbf{h}\oplus\mathrm{span}\{v\}$, which is in contradiction to the fact that the norm is smooth. Consequently, every converging subsequence of $(\alpha_n)$ goes to $\alpha$. This, together with the fact that $(\alpha_n)$ is bounded, yields that $\alpha_n\rightarrow\alpha$.

Thus, we estimate
\begin{align*} |\mathcal{L}(x_n)\cdot v - \mathcal{L}(x)\cdot v| = |\mathcal{L}(x_n)\cdot(\alpha_nx_n+z_n) - \mathcal{L}(x)\cdot(\alpha x+z)| = |\alpha_n - \alpha|,
\end{align*}
and the latter goes to $0$ as $n\rightarrow \infty$. This concludes the first step. \\

\noindent\emph{Second step.} We show that $\mathcal{L}(x_n)\rightarrow\mathcal{L}(x)$ in the dual norm up to a subsequence. Consider the sequence of the (continuous) functions $\mathcal{F} = (\mathcal{L}(x_n)|_B)_{n\in\mathbb{N}}$, where $|_B$ means that we are restricting  the domain to $B$. Since $||\mathcal{L}(x_n)||_* = ||x_n|| = 1$ for every $n\in\mathbb{N}$, we have that this family is uniformly bounded. We also have that $\mathcal{F}$ is equicontinuous, because
\begin{align*} |\mathcal{L}(x_n)\cdot v - \mathcal{L}(x_n)\cdot w| \leq ||\mathcal{L}(x_n)||_*||v-w|| = ||v-w||,
\end{align*}
for any $n \in \mathbb{N}$ and every $v,w \in B$. Noticing that $B$ is compact, the Arzela-Ascoli theorem gives that there exists a subsequence $(\mathcal{L}(x_{n_k})|_B)_{k\in\mathbb{N}}$ such that $\mathcal{L}(x_{n_k})|_B\rightarrow \phi$ in the dual norm for some continuous function $\phi\colon B\rightarrow\mathbb{R}$. Since the convergence in the dual norm implies pointwise convergence, it follows from the first step of the proof that $\phi = \mathcal{L}(x)|_B$. Hence $\mathcal{L}(x_{n_k})\rightarrow\mathcal{L}(x)$ in the dual norm. \\

\noindent\emph{Third step.} We prove that $\mathcal{L}(x_n)\rightarrow\mathcal{L}(x)$ with respect to $||\cdot||_*$ (that is, we can guarantee the convergence of the original sequence, without needing to pass to a subsequence). We proceed by contradiction. If this convergence does not hold, then there exist a number $\varepsilon > 0$ and a subsequence $(x_{n_j})_{j\in\mathbb{N}}$ such that $||\mathcal{L}(x_{n_j}) - \mathcal{L}(x)||_* > 2\varepsilon$ for every $j\in\mathbb{N}$. Hence there exists a sequence $(v_{n_j})$ of vectors of $B$ such that
\begin{align*} |\mathcal{L}(x_{n_j})\cdot v_{n_j} - \mathcal{L}(x)\cdot v_{n_j}| > \varepsilon.
\end{align*}
Passing to a subsequence if necessary, we can assume that $v_{n_j}\rightarrow v$ for some vector $v \in B$ (recall that $B$ is compact). However, the left-hand side of the inequality above can be made arbitrarily small as $j \rightarrow \infty$. This comes from the estimate
\begin{align*} |\mathcal{L}(x_{n_j})\cdot v_{n_j} - \mathcal{L}(x)\cdot v_{n_j}| \leq \\ \leq |\mathcal{L}(x_{n_j})\cdot v_{n_j} - \mathcal{L}(x_{n_j})\cdot v| + |\mathcal{L}(x_{n_j})\cdot v - \mathcal{L}(x)\cdot v| + |\mathcal{L}(x)\cdot v - \mathcal{L}(x)\cdot v_{n_j}| \leq \\ \leq ||\mathcal{L}(x_{n_j})||_*||v_{n_j}-v|| + |\mathcal{L}(x_{n_j})\cdot v - \mathcal{L}(x)\cdot v| + ||\mathcal{L}(x)||_*||v-v_{n_j}|| = \\ = 2||v-v_{n_j}|| + |\mathcal{L}(x_{n_j})\cdot v - \mathcal{L}(x)\cdot v|,
\end{align*}
where the reader may observe that, as a consequence of the first step of the proof, the last term converges to $0$ as $j\rightarrow \infty$. 

\end{proof}

Let $((\mathbb{R}^n)^{**},||\cdot||_{**})$ be the bi-dual of $\mathbb{R}^n$, which is the dual space of $(\mathbb{R}^n)^*$, endowed with the norm $||\cdot||_{**}:=(||\cdot||_*)_*$, that is, the dual norm of $||\cdot||_*$. Let $J\colon\mathbb{R}^n\rightarrow(\mathbb{R}^n)^{**}$ be the canonical identification given by $J(x)\cdot \phi = \phi(x)$, for any $x \in \mathbb{R}^n$ and every $\phi \in (\mathbb{R}^n)^*$. It is well-known that $J$ is a norm-preserving isomorphism. Next we discuss the duality of the Legendre transform.

\begin{prop} Denote by $\mathcal{L}^*$ the Legendre transform of $((\mathbb{R}^n)^*,||\cdot||_*)$ onto $(\mathbb{R}^n)^{**}$, and let $\phi \in (\mathbb{R}^n)^*$. Then $\mathcal{L}^*(\phi) = J(x)$ if and only if $\phi = \mathcal{L}(x)$. In other words, $\mathcal{L}^* = J\circ\mathcal{L}^{-1}$.
\end{prop}
\begin{proof} Assume first that $\mathcal{L}^*(\phi) = J(x)$. Since $J$ is linear, and the Legendre transform is always homogeneous of degree one, we may assume that $||x|| = 1$ (observe also that the case $x = 0$ is trivial). Noticing that $\mathcal{L}^*$ and $J$ are norm-preserving, we have
\begin{align*}\phi(x) = J(x)\cdot\phi = \mathcal{L}^*(\phi)\cdot\phi = ||\phi||_*^2 = ||\mathcal{L}^*(\phi)||_{**}^2 = ||J(x)||^2_{**} = ||x||^2 = 1.
\end{align*}
It follows that 
\begin{align*} 1 = \phi(x) = ||\phi||_* = \sup\{\phi(y):y \in B\}, 
\end{align*}
that is, the dual norm of $\phi$ is attained at $x$. In particular, we have that $\mathcal{L}(x)\cdot x = \phi(x)$. For any $z \in \mathrm{ker}(\phi)$ the inequality
\begin{align*} 1 \geq \phi\!\left(\frac{x+tz}{||x+tz||}\right) = \frac{1}{||x+tz||}
\end{align*}
holds for any $t \in \mathbb{R}$. Thus, $||x+tz|| \geq 1 = ||x||$ for every $t \in \mathbb{R}$, meaning that $x \dashv_B z$. Denoting by $\mathbf{h}$ the hyperplane which supports $B$ at $x$, it follows that $\mathrm{ker}(\phi) \subseteq \mathbf{h}$. Since both are $(n-1)$-dimensional vector subspaces, we get that $\mathrm{ker}(\phi) = \mathbf{h}$, and this proves that $\phi = \mathcal{L}(x)$. 

Now assume that $\mathcal{L}(x) = \phi$, and still consider that $||x|| = 1$. Since the Legendre transform is norm-preserving, we have $||\phi||_* = ||\mathcal{L}(x)||_* = ||x|| = 1$. Also,
\begin{align*} J(x)\cdot\phi = \phi(x) = \mathcal{L}(x)\cdot x = ||x||^2 = 1.
\end{align*}
On the other hand,
\begin{align*} \mathcal{L}^*(\phi)\cdot\phi = ||\phi||^2_* = 1,
\end{align*}
from where we get that $J(x)\cdot\phi = \mathcal{L}^*(\phi)\cdot\phi$. It remains to prove that $J(x)$ and $\mathcal{L}^*(\phi)$ have the same kernel. If $\psi \in \mathrm{ker}(J(x))$, then $0 = J(x)\cdot\psi = \psi(x)$, and this leads to
\begin{align*} ||\phi+t\psi||_* \geq \phi(x)+t\psi(x) = \phi(x) = 1 = ||\phi||_*
\end{align*}
for any $t \in \mathbb{R}$. Thus, $\psi$ is a vector of the supporting hyperplane of the dual unit ball $B^*$ at $\phi$, and hence $\psi \in \mathrm{ker}(\mathcal{L}^*(\phi))$. It follows that $\mathrm{ker}(J(x))\subseteq\mathrm{ker}(\mathcal{L}^*(\phi))$, and using again the fact that both are $(n-1)$-dimensional linear subspaces, we obtain that the equality holds. Therefore, $\mathcal{L}^*(\phi) = J(x)$. 

\end{proof}

\begin{remark} What we have proved is that, up to the canonical identification between $\mathbb{R}^n$ and $(\mathbb{R}^n)^{**}$, the Legendre transform is self-dual. Notice that the proof relies heavily on the fact that the kernel of a non-zero functional is precisely the hyperplane which supports the unit ball in the boundary point where its dual norm is attained.
\end{remark}

\begin{coro} The inverse of the Legendre transform of a smooth and strictly convex normed space is continuous.
\end{coro}
\begin{proof} Simply observe that $\mathcal{L}^{-1} = J^{-1}\circ\mathcal{L}^*$, and the latter is the composition of continuous maps ($\mathcal{L}^*$ is continuous because it is a Legendre transform).

\end{proof}

Next we give a characterization of the norm gradient of a differentiable convex function. 

\begin{teo}\label{legendregrad} Let $f\colon\mathbb{R}^n\rightarrow\mathbb{R}$ be a convex function differentiable at a point $x \in \mathbb{R}^n$. Then
\begin{align}\label{graddefiequiv} f(y) - f(x) \geq \mathcal{L}(\nabla f(x))\cdot(y-x)
\end{align}
for any $y \in \mathbb{R}^n$. The converse is also true: the norm gradient $\nabla f(x)$ is the unique vector for which this inequality holds for every $y \in \mathbb{R}^n$.
\end{teo}
\begin{proof} We have to use some machinery from the theory of convex functions. First of all, we state and prove an inequality which ``captures" the convexity in terms of one-sided derivatives. We have
\begin{align}\label{ineqconvex1} f(x+w) - f(x) \geq f'_+(x,w)
\end{align}
for any $x,w \in \mathbb{R}^n$. To prove this inequality, observe that for any $\varepsilon > 0$ the inequality
\begin{align*} f(x+\varepsilon w) -f(x) = f(\varepsilon(x+w)+(1-\varepsilon)x) -f(x) \leq \varepsilon f(x+w) + (1-\varepsilon)f(x) -f(x) = \\ = \varepsilon (f(x+w)-f(x))
\end{align*}
holds, yielding
\begin{align*} f(x+w)-f(x) \geq \frac{f(x+\varepsilon w)-f(x)}{\varepsilon}.
\end{align*}
Hence, letting $\varepsilon \rightarrow 0^+$ we get (\ref{ineqconvex1}).

For simplicity of notation, we write $v = \nabla f(x)$ and $c = f(x)$. First assume that $c$ is not a global minimum of $f$. Then $v \neq 0$, and there is a unique hyperplane $\mathbf{h}$ supporting $\{f\leq c\}$ at $x$. Consequently, we have the direct sum

\begin{align*} \mathbb{R}^n = \mathbf{h} \oplus \mathrm{span}\{v\},
\end{align*}
and since any translation is bijective, we get that any point $y \in \mathbb{R}^n$ can be written in the form
\begin{align*} y = \lambda v + x + z,
\end{align*}
for some $\lambda \in \mathbb{R}$ and some $z \in \mathbf{h}$. From inequality (\ref{ineqconvex1}) and the linearity of $f'_+(x,\cdot)$ we get
\begin{align}\label{ineqconvex2} f(y) - f(x) \geq f'_+(x,\lambda v+z) = \lambda f'_+(x,v) + f'_+(x,z).
\end{align}
Since $f'_+(x,\cdot)$ is the Fr\'{e}chet derivative of $f$ at $x$, we get from the definition of the norm gradient that
\begin{align*} f'_+(x,v) = df_x(v) = ||v||^2.
\end{align*}
On the other hand, since $z \in \mathbf{h}$ we have that, for any $\varepsilon \in \mathbb{R}$, the point $x+\varepsilon z$ lies in the hyperplane which supports $\{f\leq c\}$ at $x$. It follows that $f(x+\varepsilon z) \geq c = f(x)$ for any $\varepsilon > 0$, and this leads to
\begin{align*} f'_+(x,z) = \lim_{\varepsilon\rightarrow 0^+}\frac{f(x+\varepsilon z)-f(x)}{\varepsilon} \geq 0.
\end{align*}
Plugging this inequality in (\ref{ineqconvex2}), we get
\begin{align}\label{ineqconvex3} f(y) - f(x) \geq \lambda f'_+(x,v) = \lambda||v||^2.
\end{align}
Now we observe that
\begin{align*} \mathcal{L}(v)\cdot(y-x) = \mathcal{L}(v)\cdot(\lambda v + z) = \lambda\mathcal{L}(v)\cdot v + \mathcal{L}(v)\cdot z\,.
\end{align*}
From the definition of the Legendre transform we have that $\mathcal{L}(v)\cdot z = 0$, since $z$ is a vector of the hyperplane which is Birkhoff right-orthogonal, and $\mathcal{L}(v)\cdot v = ||v||^2$. Consequently, $\mathcal{L}(v)\cdot(y-x) = \lambda||v||^2$. Together with (\ref{ineqconvex3}) this gives the desired inequality.

Now we assume that $c = f(x)$ is a global minimum of $f$, and we have $\nabla f(x) = 0$ by the definition. As a consequence, we get that $f(y) - f(x) \geq 0 = \mathcal{L}(\nabla f(x))\cdot (y-x)$ for any $y \in \mathbb{R}^n$.

For the converse, suppose first that $c = f(x)$ is not the global minimum of $f$, and assume that $v \in \mathbb{R}^n$ is such that $f(y) - f(x) \geq \mathcal{L}(v)\cdot(y-x)$. Notice that we must have $v \neq 0$ (otherwise $c$ is a global minimum), and let $\mathbf{h}$ be the hyperplane which is Birkhoff right-orthogonal to $v$, translated to pass through $x$. We claim that $\mathbf{h}$ supports the sub-level set $\{f \leq c\}$ at $x$. Indeed, if this is not true, then we can take a point $y \in \mathbf{h}\cap\mathrm{int}\{f\leq c\} = \mathbf{h}\cap\{f < c\}$. It follows that $v \dashv_B y-x$, which means that
\begin{align*} \mathcal{L}(v)\cdot(y-x) = 0.
\end{align*}
Thus,
\begin{align*} f(y) - f(x) < c - f(x) = c-c = 0 = \mathcal{L}(v)\cdot(y-x),
\end{align*}
and this contradiction shows that $v \dashv_B \mathbf{h}$. It still remains to prove that $df_x(v) = ||v||^2$. First we notice that for any $\lambda > 0$ we have
\begin{align*} f(x+\lambda x) - f(x) \geq \mathcal{L}(v)\cdot(\lambda v) = \lambda ||v||^2,
\end{align*}
from which we get
\begin{align*} df_x(v) = f'_+(x,v) = \lim_{\lambda \rightarrow 0^+}\frac{f(x+\lambda v)-f(x)}{\lambda} \geq ||v||^2.
\end{align*}
For the reverse inequality we observe that with $\lambda > 0$ also
\begin{align*} f(x-\lambda v) - f(x) \geq \mathcal{L}(v)\cdot(-\lambda v) = -\lambda||v||^2,
\end{align*}
holds, and hence
\begin{align*} ||v||^2 \geq \lim_{\lambda \rightarrow 0^+}\frac{f(x-\lambda v)-f(x)}{-\lambda} = df_x(v).
\end{align*}

Finally, assume that $c = f(x)$ is a global minimum of $f$, and let $v \in \mathbb{R}^n$ be a vector such that $f(y) - f(x) \geq \mathcal{L}(v)\cdot(y-x)$ for every $y \in \mathbb{R}^n$. Since $c$ is a global minimum, we have that $df_x = 0$, and then, in particular, $df_x(v) = 0$. For any $\lambda > 0$ we have
\begin{align*} f(x+\lambda v) - f(x) \geq \mathcal{L}(v)\cdot(\lambda v) = \lambda||v||^2,
\end{align*}
and, similarly to what we have done before, we obtain
\begin{align*} ||v||^2\leq \lim_{\lambda\rightarrow 0^+}\frac{f(x+\lambda v)-f(x)}{\lambda} = f'_+(x,v) = df_x(v) = 0,
\end{align*}
implying also that $v = 0$. The proof is complete.

\end{proof}

Notice carefully that inequality (\ref{graddefiequiv}) provides an equivalent definition for the norm gradient of a convex function $f$ in a point $x$ where $f$ is differentiable. Inspired by that, we extend this notion for a convex function $f\colon \mathbb{R}^n\rightarrow \mathbb{R}$ which is not necessarily differentiable. We say that $v \in \mathbb{R}^n$ is a \emph{norm sub-gradient} of $f$ at $x$ if
\begin{align}\label{normsubgradientdef} f(y) - f(x) \geq \mathcal{L}(v)\cdot(y-x)
\end{align}
for any $y \in \mathbb{R}^n$. For each $x \in \mathbb{R}^n$, the set $\partial f(x)$ of the norm sub-gradients of $f$ at $x$ is called the \emph{norm sub-differential} of $f$ at $x$. It is immediate to check that $\partial f(x)$ is always closed. It is also easy to see that $f(x) = c$ is a global minimum of $f$ if and only if $0 \in \partial f(x)$. At this point, we warn the reader that other generalizations of the concept of sub-gradient were studied, also in view of differentiability properties (see, e.g., \cite{borwein2} and \cite{clarke}). As in the Euclidean case, we have a characterization of the norm sub-gradients in terms of one-sided derivatives (see \cite{borwein}, for example).

\begin{lemma}\label{maxlemma} Let $f\colon\mathbb{R}^n\rightarrow\mathbb{R}$ be a convex function, and let $x \in \mathbb{R}^n$. We have that $v \in \partial f(x)$ if and only if
\begin{align*} f'_+(x,u) \geq \mathcal{L}(v)\cdot u
\end{align*}
for every $u \in \mathbb{R}^n$.
\end{lemma}
\begin{proof} Assume first that $v \in \partial f(x)$. Then for each $u \in \mathbb{R}^n$ and any $\lambda > 0$ we have that
\begin{align*} f(x+\lambda u) - f(x) \geq \mathcal{L}(v)\cdot\lambda u = \lambda \mathcal{L}(v)\cdot u.
\end{align*}
It follows that
\begin{align*} \frac{f(x+\lambda u)-f(x)}{\lambda} \geq \mathcal{L}(v)\cdot u,
\end{align*}
for any $\lambda > 0$. Letting $\lambda \rightarrow 0^+$, we have the desired inequality.

Now assume that $v \in \mathbb{R}^n$ is a vector such that $f'_+(x,u) \geq \mathcal{L}(v)\cdot u$. If $y \in\mathbb{R}^n$, then from inequality (\ref{ineqconvex1}) we get
\begin{align*} f(y) - f(x) \geq f'_+(x,y-x) \geq \mathcal{L}(v)\cdot(y-x),
\end{align*}
and this concludes the proof.

\end{proof}

\begin{teo}\label{maxtheorem} If $f\colon\mathbb{R}^n\rightarrow\mathbb{R}$ is a convex function, then for any $x \in \mathbb{R}^n$ we have that $\partial f(x) \neq \emptyset$ and
\begin{align*} f'_+(x,u) = \max\{\mathcal{L}(w)\cdot u:w \in \partial f(x)\},
\end{align*}
for each $u \in \mathbb{R}^n$. 
\end{teo}
\begin{proof} By the previous lemma, we already have that the inequality
\begin{align*} \max\{\mathcal{L}(w)\cdot u:w \in \partial f(x)\} \leq f'_+(x,u)
\end{align*}
holds for any $u \in \mathbb{R}^n$. Hence we only have to show that for an arbitrarily given vector $u \in \mathbb{R}^n$ there exists a vector $w \in \partial f(x)$ such that $\mathcal{L}(w)\cdot u = f'_+(x,u)$. 

To prove that, we first assume that $u \neq 0$. Choose a basis $\{e_1,\ldots,e_n\}$ of $\mathbb{R}^n$ with the property that $e_1 = u$. Let $g_0(\cdot):=f'_+(x,\cdot)$, and define recursively functions $g_1,\ldots,g_n$ by setting $g_m(\cdot) := (g_{m-1})'_+(e_m,\cdot)$. Each of these functions (including $g_0$) is sub-linear, and from property \textbf{(d)} of Lemma \ref{sublinearproperties} we have that
\begin{align*} \mathrm{lin}(g_m) \supseteq \mathrm{lin}(g_{m-1}) + \mathrm{span}(e_m)
\end{align*}
for each $m = 1,\ldots,n$. From property \textbf{(e)} of Lemma \ref{sublinearproperties} we have that $g_n$ is a linear functional. Since the Legendre transform is a bijection, there exists $w \in \mathbb{R}^n$ such that $g_n = \mathcal{L}(w)$. We claim that $w$ is a norm sub-gradient such that $\mathcal{L}(w)\cdot u = f'_+(x,u)$. To check that, we first observe that by property \textbf{(b)} of Lemma \ref{sublinearproperties} the inequalities $g_0 \geq g_1\geq\ldots \geq g_n$ hold. This yields
\begin{align*} \mathcal{L}(w)\cdot(y-x) = g_n(y-x) \leq g_0(y-x) = f'_+(x,y-x) \leq f(y) - f(x),
\end{align*}
for any $x,y \in \mathbb{R}^n$, where the last inequality comes from (\ref{ineqconvex1}). This shows that $w \in \partial f(x)$ and, in particular, we also get that $\partial f(x) \neq \emptyset$ (notice that this construction holds true for any $u \neq 0$). Finally, from properties \textbf{(a)} and \textbf{(b}) of Lemma \ref{sublinearproperties} it follows that
\begin{align*} g_n(u) \leq g_0(u) = -(g_0)' _+(u,-u) = -(g_0)'_+(e_1,-u) = -g_1(-u) \leq -g_n(-u) = g_n(u),
\end{align*}
from where we get that $\mathcal{L}(w)\cdot u = g_n(u) = g_0(u) = f'_+(x,u)$. This concludes the proof of the case $u \neq 0$. If $u = 0$, then any $w \in \partial f(x)$ satisfies $\mathcal{L}(w)\cdot u = 0 = f'_+(x,u)$. Since we already have that $\partial f(x) \neq \emptyset$, the proof is finished. 

\end{proof}

\begin{remark}\label{ident} This proof appears for the Euclidean sub-case in \cite[Theorem 3.1.8]{borwein}. The reader may notice that the proof constructs a linear functional rather than a sub-gradient. Only after that, we identify this linear functional with a vector. In the classical theory, this is made via an inner product, and in our case we use the Legendre transform. 
\end{remark}

The next and important corollary is also an immediate analogue of the Euclidean sub-case. It states that we can ``detect" differentiability of a convex point at a given interior point of its domain looking at the norm sub-differential at this point.

\begin{coro}\label{singletondiff} A convex function $f\colon\mathbb{R}^n\rightarrow\mathbb{R}$ is differentiable at $x\in \mathbb{R}^n$ if and only if $\partial f(x)$ is a singleton. In this case, the unique norm sub-gradient is the norm gradient, and we have
\begin{align*} df_x(u) = \mathcal{L}(\nabla f(x))\cdot u
\end{align*}
for every $u \in \mathbb{R}^n$.
\end{coro}
\begin{proof} The uniqueness part of Theorem \ref{legendregrad} already gives that if $f$ is differentiable at $x$, then $\nabla f(x)$ is the unique norm sub-gradient of $f$ at $x$. Hence it remains to prove the converse.

Assume that $\partial f(x) = \{v\}$. Then for any $u \in \mathbb{R}^n$ we have
\begin{align*} f'_+(x,-u) = \mathcal{L}(v)\cdot (-u) = -\mathcal{L}(w)\cdot u = -f'_+(x,u),  
\end{align*}
meaning that every $u \in \mathbb{R}^n$ is a linearity direction of $f'_+(x,\cdot)$, which is, therefore, a linear map. It follows from Proposition \ref{diffconvexchar} that $f$ is differentiable at $x$. 

To finish the proof, just notice that if $f$ is differentiable at $x\in\mathbb{R}^n$, then $\partial f(x) = \{\nabla f(x)\}$, and consequently
\begin{align*}df_x(u) = f'_+(x,u) = \max\{\mathcal{L}(w)\cdot u:w \in \partial f(x)\} = \mathcal{L}(\nabla f(x))\cdot u
\end{align*}
for each $u \in \mathbb{R}^n$. 

\end{proof}

\begin{remark}\label{continuousequiv} We have defined functions of class $C^1$ as differentiable functions whose respective Euclidean gradients are continuous. This is clearly the same as demanding that their respective differential maps are continuous as maps of $\mathbb{R}^n$ onto $(\mathbb{R}^n)^*$. Namely, this follows easily from the fact that the identification between $\mathbb{R}^n$ and $(\mathbb{R}^n)^*$ given by the standard inner product is linear (and hence continuous). Since the Legendre transform and its inverse are also continuous, it follows from the last corollary that a convex differentiable function is of class $C^1$ if and only if its norm gradient is continuous for any norm. 
\end{remark}

For the sake of completeness, we state and prove two other properties of the norm sub-differential which are completely analogous to the usual sub-differential. One of them is a way to ``detect" convexity via the norm sub-differential, and the other one is a version of Rockafellar's theorem. 

\begin{prop} A function $f\colon(\mathbb{R}^n,||\cdot||)\rightarrow\mathbb{R}$ is convex if and only if $\partial f(x) \neq \emptyset$ for every $x\in\mathbb{R}^n$. 
\end{prop}
\begin{proof} We already know that if $f$ is convex, then $\partial f(x)$ is non-empty for every $x \in \mathbb{R}^n$ (from Theorem \ref{maxtheorem}). Hence we prove the converse. Let $x,y \in \mathbb{R}^n$ and $\lambda \in [0,1]$, and observe that we may take a vector $w \in \partial f((1-\lambda)x+\lambda y)$. Therefore,
\begin{align*} f(x) - f((1-\lambda)x + \lambda y) \geq \mathcal{L}(w)\cdot(x-(1-\lambda)x-\lambda y) = \lambda\mathcal{L}(w)\cdot(x-y)
\end{align*}
and
\begin{align*} f(y) - f((1-\lambda x)+\lambda y) \geq \mathcal{L}(w)\cdot(y - (1-\lambda)x - \lambda y) = (1-\lambda)\mathcal{L}(w)\cdot(y-x).
\end{align*}
Multiplying the first inequality by $(1-\lambda)$, the second by $\lambda$, and adding both, we get that
\begin{align*} f((1-\lambda)x+\lambda y) \leq (1-\lambda)f(x) + \lambda f(y),
\end{align*}
and this proves that $f$ is convex. 

\end{proof}

In what follows, we define the \emph{norm sub-differential} of a convex function $f$ as the set $\partial f\subseteq \mathbb{R}^n\times\mathbb{R}^n$ given by
\begin{align*} \partial f := \{(x,w) \in \mathbb{R}^n\times\mathbb{R}^n:w \in \partial f(x)\}. 
\end{align*}
A subset $S\subseteq \mathbb{R}^n\times\mathbb{R}^n$ is said to be \emph{norm cyclically monotonic} if for every $m \in \mathbb{N}$ and any subset $\{(x_0,w_0),\ldots,(x_m,w_m)\} \in S$ the number
\begin{align*} \mathcal{L}(w_0)\cdot(x_1-x_0) + \mathcal{L}(w_1)\cdot(x_2-x_1)+\ldots+ \mathcal{L}(w_{m-1})\cdot(x_m-x_{m-1}) + \mathcal{L}(w_m)\cdot(x_0-x_m)
\end{align*}
is non-positive. If $f$ is a convex function, then it is clear that any finite ordered subset $\{(x_0,w_0)\ldots,(x_m,w_m)\}$ of $\partial f$ is norm cyclically monotonic, because
\begin{align*}\mathcal{L}(w_0)\cdot(x_1-x_0) + \mathcal{L}(w_1)\cdot(x_2-x_1)+\ldots+ \mathcal{L}(w_{m-1})\cdot(x_m-x_{m-1}) + \mathcal{L}(w_m)\cdot(x_0-x_m) \leq \\ \leq f(x_1) - f(x_0) + f(x_2) - f(x_1) + \ldots +f(x_{m}) - f(x_{m-1})+ f(x_0) - f(x_m) = 0.
\end{align*}

The version of Rockafellar's theorem that we will prove next states, roughly speaking, that any norm cyclically monotonic set is contained in the norm sub-differential of some convex function. But for that sake we need to extend (in the natural way) the definition of convex functions to functions which take values on the extended real line $(-\infty,+\infty]$. A convex function $f\colon\mathbb{R}^n\rightarrow(-\infty,+\infty]$ is said to be \emph{proper} if $\{f = \infty\} \neq \mathbb{R}^n$. 

\begin{teo} A non-empty set $S\subseteq\mathbb{R}^n\times\mathbb{R}^n$ is norm cyclically monotonic if and only if there exists a proper convex function $f\colon\mathbb{R}^n\rightarrow(-\infty,+\infty]$ such that $S\subseteq \partial f$.
\end{teo}
\begin{proof} We already know that if $S\subseteq\partial f$ for some proper convex function $f$, then $S$ is norm cyclically monotonic. Thus, we prove the converse. Let $S\subseteq \mathbb{R}^n\times\mathbb{R}^n$ be norm cyclically monotonic. We fix $(x_0,w_0) \in S$ and define $f\colon\mathbb{R}^n\rightarrow(-\infty,+\infty]$ by
\begin{align*} f(x) = \sup\{\mathcal{L}(w_m)\cdot(x-x_m) +\mathcal{L}(w_{m-1})\cdot(x_m-x_{m-1}) + \ldots + \mathcal{L}(w_0)\cdot(x_1-x_0)\},
\end{align*}
where the supremum is taken over all values of $m\in\mathbb{N}$ and every possible choice of points $(x_j,w_j) \in S$ for $j = 1,\ldots,m$. The function $f$ defined this way is the supremum of affine functions, and hence it is a convex function (see \cite[Chapter 1]{schneider}). Moreover, since $S$ is norm cyclically montonic, it follows that $f(x_0) = 0$, and then $f$ is proper. 

Now let $(x,w) \in S$. We have to prove that $w$ is a norm sub-gradient of $f$ at $x$. For this purpose, choose a number $\alpha < f(x)$. Hence there exist pairs $(x_1,w_1),\ldots,(x_m,w_m)$ such that
\begin{align*} \alpha < \mathcal{L}(w_m)\cdot(x-x_m) + \mathcal{L}(w_{m-1})\cdot(x_m-x_{m-1}) + \ldots + \mathcal{L}(w_0)\cdot(x_1-x_0).
\end{align*}
Putting $(x,w) = (x_{m+1},w_{m+1})$, it also comes from the definition of $f$ that
\begin{align*} f(y) \geq \mathcal{L}(w_{m+1})\cdot(y-x_{m+1}) + \mathcal{L}(w_m)\cdot(x_{m+1}-x_m) + \ldots + \mathcal{L}(w_0)\cdot(x_1-x_0),
\end{align*}
for every $y \in \mathbb{R}^n$. These two inequalities yield immediately that
\begin{align*} f(y) > \alpha + \mathcal{L}(w_{m+1})\cdot(y-x_{m+1}) = \alpha + \mathcal{L}(w)\cdot(y-x).
\end{align*}
Since this holds for any $\alpha < f(x)$ and every $y \in \mathbb{R}^n$, we have that
\begin{align*} f(y) \geq f(x) + \mathcal{L}(w)\cdot(y-x)
\end{align*} 
for each $y \in \mathbb{R}^n$. It follows that $w \in \partial f(x)$, and the proof is done. 

\end{proof}

\begin{remark} The proof of the Euclidean case is given in \cite[Theorem 1.5.16]{schneider}. 
\end{remark}

Here a disclaimer is due, namely on reasons why all of that ``works so well" when we only have a (smooth and strictly convex) norm to work with. The idea behind sub-gradients is to ``control" the one-sided derivatives using linear functionals. As it is pointed out in Remark \ref{ident}, in the classical theory these functionals are simply identified with vectors via the standard inner product of $\mathbb{R}^n$. When we have a norm, we still can identify vectors with linear functionals in a way which is coherent with the norm of the domain by using the Legendre transform. 

Things get more interesting when we adopt the geometric point of view. In the classical theory, sub-differentials are related to normal cones of sub-level sets. The natural question that arises is whether this is also true in our case, when one replaces the standard normality notion (combined with the  
standard inner product) by Birkhoff orthogonality. The next theorems are devoted to answer this question. 

\begin{teo}\label{normsubgradgeom} Let $x \in \mathbb{R}^n$ and $c = f(x)$. Assume that $c$ is not a global minimum of $f$, and let $v \in\mathbb{R}^n$ be a norm sub-gradient of $f$ at $x$. Then the hyperplane $\mathbf{h}$ which is Birkhoff right-orthogonal to $v$ supports $\{f\leq c\}$ at $x$. Moreover, we have the estimates
\begin{align}\label{estimatesub}  \sup\{f'_-(x,v+z): z\in\mathbf{h}\} \leq ||v||^2 \leq \inf\{f'_+(x,v+z): z\in\mathbf{h}\}.
\end{align}
\end{teo}
\begin{remark}\label{nonzerov} If $v \neq 0$ and \textbf{(i)} holds, then one can easily see that $f'_-(x,v)$ and $f'_+(x,v)$ are positive numbers.
\end{remark}
\begin{proof} By abuse of notation, we assume that $\mathbf{h}$ is translated to pass through $x$. If $\mathbf{h}$ does not support $\{f\leq c\}$ at $x$, then one may take a point $y \in \mathbf{h}\cap\{f < c\}$. Since $y - x \in \mathbf{h}$ and $v \dashv_B \mathbf{h}$, we have that $\mathcal{L}(v)\cdot(y-x) = 0$. Once $f(y) < c = f(x)$, we get $f(y) - f(x) < 0 = \mathcal{L}(v)\cdot(y-x)$, which contradicts the fact that $v$ is a sub-gradient of $f$ at $x$. This proves that $\mathbf{h}$ supports $\{f\leq c\}$ at $x$.

The estimates come from Lemma \ref{maxlemma}. For any $z \in \mathbf{h}$ we have that
\begin{align*} f'_+(x,v+z) \geq \mathcal{L}(v)\cdot(v+z) = \mathcal{L}(v)\cdot v = ||v||^2 \ \ \mathrm{and}\\
f'_-(x,v+z) = -f'_+(x,-v-z) \leq -\mathcal{L}(v)\cdot(-v-z) = ||v||^2,
\end{align*}
and this concludes the proof.

\end{proof}

The natural question that arises here is whether the converse holds true, namely the following implication: if a vector $v$ is Birkhoff left-orthogonal to some hyperplane which supports $\{f \leq c\}$ at $x$, and if (\ref{estimatesub}) also holds, is $v$ then a norm sub-gradient of $f$? To answer that question (positively), we first need to check that the estimates (\ref{estimatesub}) really ``make sense".

\begin{prop}\label{regular} Let $f\colon\mathbb{R}^n\rightarrow\mathbb{R}$ be a convex function, and assume that $f(x) = c$ is not a global minimum. If $\mathbf{h}$ is a supporting hyperplane of $\{f\leq c\}$ at $x$, and if $v \dashv_B \mathbf{h}$, then
\begin{align*} f'_-(x,v+z_1) \leq f'_+(x,v+z_2)
\end{align*}
for any $z_1,z_2 \in \mathbf{h}$. As a consequence, we have that
\begin{align*} \sup\{f'_-(x,v+z): z\in\mathbf{h}\} \leq \inf\{f'_+(x,v+z): z\in\mathbf{h}\}.
\end{align*}
\end{prop}
\begin{proof} Suppose that there exist $z_1,z_2 \in \mathbf{h}$ such that
\begin{align*} f'_-(x,v+z_1) > f'_+(x,v+z_2).
\end{align*}
This yields
\begin{align*} f'_+(x,z_2-z_1) \leq f'_+(x,v+z_2) + f'_+(x,-v-z_1) = f'_+(x,v+z_2) - f'_-(x,v+z_1) < 0,
\end{align*}
and this is a contradiction because $z_2 - z_1 \in \mathbf{h}$. Indeed, since $\mathbf{h}$ supports $\{f\leq c\}$ at $x$, we have that, at $x$, $f$ is non-decreasing in the direction of $z_2-z_1$. 

\end{proof}

It follows from Proposition \ref{regular} that if $\mathbf{h}$ supports $\{f\leq c\}$ at $x$, then we always can choose a vector $v\dashv_B \mathbf{h}$ for which (\ref{estimatesub}) holds. Next we will prove that such a vector is a norm sub-gradient. This guarantees the existence of norm sub-gradients pointing in all Birkhoff outer normal directions of $\{f\leq c\}$ at $x$.

\begin{teo}\label{subgradientchar} Let $f\colon\mathbb{R}^n\rightarrow\mathbb{R}$ be a convex function. As usual, assume that $f(x) = c$ is not a global minimum, and that the hyperplane $\mathbf{h}$ supports $\{f\leq c\}$ at $x$. If $v \in\mathbb{R}^n$ is a vector such that\\

\noindent\textbf{\emph{(i)}} $v \dashv_B \mathbf{h}$ and\\

\noindent\textbf{\emph{(ii)}} $\sup\{f'_-(x,v+z): z\in\mathbf{h}\} \leq ||v||^2 \leq \inf\{f'_+(x,v+z): z\in\mathbf{h}\}$,\\

\noindent then $v$ is a norm sub-gradient of $f$ at $x$. 
\end{teo}
\begin{proof} First observe that, by Remark \ref{nonzerov}, $v$ is non-zero. We have to prove that $f(y) - f(x) \geq \mathcal{L}(v)\cdot(y-x)$. Any $y \in \mathbb{R}^n$ can be written as $y = x + z + \lambda v$ for some $\lambda \in \mathbb{R}$ and some $z \in \mathbf{h}$. Assuming first that $\lambda > 0$, we have
\begin{align*} f(y) - f(x) = f(x+z+\lambda v) - f(x) \geq f'_+(x,z+\lambda v) = \lambda f'_+\Big(x,v+\frac{z}{\lambda}\Big) \geq \lambda||v||^2 =\\ = \mathcal{L}(v)\cdot(\lambda v) = \mathcal{L}(v)\cdot(z+\lambda v) = \mathcal{L}(v)\cdot(y-x),
\end{align*}
where the first inequality comes from (\ref{ineqconvex1}). We also used the definition of the Legendre transform. If $\lambda < 0$, we have
\begin{align*} f(y) - f(x) \geq f'_+(x,z+\lambda v) = -\lambda f'_+\Big(x,-v-\frac{z}{\lambda}\Big) = \lambda f'_-\Big(x,v+\frac{z}{\lambda}\Big) \geq \lambda ||v||^2 = \\ = \mathcal{L}(v)\cdot(z+\lambda v) = \mathcal{L}(v)\cdot(y-x).
\end{align*}
It only remains to prove the case $\lambda = 0$. To do so, we recall that $f'_+(x,z) \geq 0$, because $z$ is a vector of a supporting hyperplane of $\{f\leq c\}$ at $x$. Hence
\begin{align*} f(y) - f(x) = f(x+z) - f(x) \geq f'_+(x,z) \geq 0 = \mathcal{L}(v)\cdot z = \mathcal{L}(v)\cdot(y-x),
\end{align*}
and the proof is complete. 

\end{proof}

\begin{coro} Under the same conditions as in the previous theorem, let $u$ be a unit outward pointing vector Birkhoff orthogonal to a supporting hyperplane $\mathbf{h}$ of $\{f\leq c\}$ at $x$. If
\begin{align*} \sup\{f'_-(x,u+z):z\in\mathbf{h}\} \leq \lambda \leq \inf\{f'_+(x,u+z):z\in\mathbf{h}\},
\end{align*}
then $\lambda u \in \partial f(x)$. 
\end{coro}
\begin{proof} We have to prove that conditions \textbf{(i)} and \textbf{(ii)} of Theorem \ref{subgradientchar} hold. The first is obvious, because Birkhoff orthogonality is homogeneous. For the second, recall that $\lambda > 0$ (see Remark \ref{nonzerov}), notice that 
\begin{align*} \lambda\cdot\sup\{f'_-(x,u+z):z\in\mathbf{h}\} = \sup\{f'_-(x,\lambda u+z):z\in\mathbf{h}\},
\end{align*}
and that the same holds for the infimum. Then
\begin{align*} \sup\{f'_-(x,\lambda v+z):z\in\mathbf{h}\} \leq \lambda^2 = ||\lambda u||^2 \leq \inf\{f'_+(x,\lambda u+z):z\in\mathbf{h}\},
\end{align*}
which is condition \textbf{(ii)} of Theorem \ref{subgradientchar}. 

\end{proof}

As we shall see later, Theorem \ref{normsubgradgeom} is a key ingredient in proving that a distance function to a convex body in a (smooth) normed space is differentiable outside the body. But this can also be used to give an easy proof to a well-known result regarding the regularity of norms.

\begin{teo} Let $\rho\colon\mathbb{R}^n\rightarrow\mathbb{R}$ be a norm. If its unit ball $B$ is smooth, then $\rho$ is $C^1$ on $\mathbb{R}^n\setminus\{0\}$. Moreover, its norm gradient is given by
\begin{align*} \nabla\rho(x) = \frac{x}{\rho(x)},
\end{align*}
for each $x \in \mathbb{R}^n\setminus\{0\}$. 
\end{teo}
\begin{proof} First notice that $\rho$ is a sub-linear function, and hence it is convex. Thus, to show that $\rho$ is differentiable at a given point $x$, it suffices to prove that the norm sub-differential of $\rho$ at $x$ contains a unique element. Let $x \in \mathbb{R}^n\setminus\{0\}$ be such that $\rho(x) = c$ ($> 0$). The sub-level set $\{\rho \leq c\}$ is the ball $cB$, which by the smoothness hypothesis is supported at $x$ by a unique hyperplane ($\mathbf{h}$, say). Thus, given $v \in \partial\rho(x)$ we have $v \dashv_B \mathbf{h}$, meaning that $v = \beta x$ for some $\beta \in \mathbb{R}$. Since
\begin{align*} \left.\frac{d}{dt}\rho(x\pm tx)\right|_{t=0} = \pm\rho(x),
\end{align*}
we have that $\rho'_+(x,x) = \rho(x) = -\rho'_+(x,-x)$. With Lemma \ref{maxlemma} this yields
\begin{align*} \rho(x) = \rho'_+(x,x) \geq \mathcal{L}(v)\cdot x = \mathcal{L}(\beta x)\cdot x = \beta\rho(x)^2
\end{align*}
and
\begin{align*} -\rho(x) = \rho'_+(x,-x) \geq \mathcal{L}(v)\cdot(-x) = -\mathcal{L}(\beta x)\cdot x = -\beta\rho(x)^2.
\end{align*}
These inequalities give that $\beta = 1/\rho(x)$. It follows that $\partial\rho(x)$ is a singleton, and hence $\rho$ is differentiable. The unique element of $\partial\rho(x)$ is the norm gradient
\begin{align*} \nabla\rho(x) = \beta x = \frac{x}{\rho(x)},
\end{align*}
which is clearly continuous. It follows that $\rho$ is $C^1$ on $\mathbb{R}^n\setminus\{0\}$. 

\end{proof}

\begin{remark} Notice that by the continuity of the Legendre transform we have that if the norm gradient is continuous (as a map onto $\mathbb{R}^n$), then the Euclidean gradient is also continuous. Indeed, the Euclidean gradient is the image of the norm gradient under the composition of the Legendre transform with the isomorphism between $\mathbb{R}^n$ and $(\mathbb{R}^n)^*$ given by the standard inner product. 
\end{remark}

In some fields (such as Finsler geometry, for example) it is more common to define the Legendre transform by means of the derivative of the norm. However, in the beginning we did not assume  
differentiability of the norm, but only strict convexity and smoothness. The last theorem states that, under these hypotheses, the norm is indeed differentiable, and hence we can characterize now the Legendre transform by means of the derivative of the norm. 

\begin{coro} Let $\rho$ be a smooth norm on $\mathbb{R}^n$. The associated Legendre transform can be written as
\begin{align*} \mathcal{L}(x) = \rho(x)\cdot d\rho_x(\cdot),
\end{align*}
for each $x \in \mathbb{R}^n$.
\end{coro}
\begin{proof} If $x = 0$, then there is nothing to prove. Hence we assume that $x \neq 0$ and put $c = \rho(x)$. With $B$ as unit ball of $\rho$, let $\mathbf{h}$ be the hyperplane such that $x \dashv_B\mathbf{h}$. If $z \in \mathbf{h}$, then we may take a differentiable curve $\gamma(t)\colon I\subseteq\mathbb{R}\rightarrow c\cdot\partial B$ such that $\gamma(0) = x$ and $\gamma'(0) = z$. Thus,
\begin{align*} \rho(x)\cdot d\rho_x(z) = \rho(x)\cdot\left.\frac{d}{dt}\rho\circ\gamma(t)\right|_{t=0} = 0 = \mathcal{L}(x)\cdot z,
\end{align*}
because $\rho\circ\gamma(t) = c$ for every $t$. Finally, we calculate
\begin{align*} \rho(x)\cdot d\rho_x(x) = \rho(x)\cdot\left.\frac{d}{dt}\rho(x+tx)\right|_{t=0} = \rho(x)^2 = \mathcal{L}(x)\cdot x.
\end{align*}

\end{proof}

\begin{remark} Having the differentiability of the norm (except at the origin, of course) \emph{a priori}, the Legendre transform can be equivalently defined as
\begin{align*} \mathcal{L}(x)\cdot v := \frac{1}{2}\left.\frac{d}{dt}\rho(x+tv)^2\right|_{t=0},
\end{align*}
for any $x,v \in \mathbb{R}^n$. This approach was taken in \cite{alvarez}, for example. Even more, in \cite{horvath} the authors study the map $[v,x] := \mathcal{L}(x)\cdot v$, which they call a \emph{semi-inner product}. 
\end{remark}

To finish this section, we relate norm sub-differentials with normal cones of sub-level sets at boundary points. Let $K$ be a convex body, and $x$ be a boundary point. The (\emph{Birkhoff}) \emph{normal cone} of $K$ at $x$, denoted by $\mathrm{NC}(K,x)$, is the set of all outward pointing vectors which are Birkhoff left-orthogonal to some hyperplane that supports $K$ at $x$, together with the zero vector (for other normal cones appearing in the geometry of Banach spaces we refer the reader to \cite{sain}). Since an outward pointing unit Birkhoff left-orthogonal vector to some supporting hyperplane of $K$ at $x$ is called an outer normal, and since Birkhoff orthogonality is homogeneous, we have
\begin{align*} \mathrm{NC}(K,x) = \mathbb{R}^+\cdot\{\mathrm{outer \ normals \ of} \ K \ \mathrm{at} \ x\},
\end{align*}
where $\mathbb{R}^+\cdot A := \{\lambda a:\lambda \geq 0 \ \mathrm{and} \ a \in A\}$. A consequence of Theorem \ref{normsubgradgeom} is that the inclusion
\begin{align*} \mathbb{R}^+\cdot \partial f(x)\subseteq \mathrm{NC}(\{f\leq c\},x)
\end{align*}
holds whenever $f\colon\mathbb{R}^n\rightarrow\mathbb{R}$ is a convex function and $f(x) = c$ is not a global minimum. Under these same hypotheses, Theorem \ref{subgradientchar} gives the reverse inclusion:
\begin{align*} \mathrm{NC}(\{f\leq c\},x) \subseteq \mathbb{R}^+\cdot \partial f(x).
\end{align*}
As it was mentioned before, the norm sub-differential is the pull-back of a (convex) set of linear functionals in $(\mathbb{R}^n)^*$ by the Legendre transform:
\begin{align*} \partial f(x) = \mathcal{L}^{-1}\Big(\{\phi \in (\mathbb{R}^n)^*:f(y) - f(x) \geq \phi(y-x) \ \ \forall \ y \in \mathbb{R}^n\}\Big),
\end{align*} 
and the same holds for the Birkhoff normal cone, as the next lemma shows.

\begin{lemma}\label{functionalcone} Let $K\subseteq\mathbb{R}^n$ be a convex body, and let $x \in \partial K$ be a boundary point. Then
\begin{align*} \mathrm{NC}(K,x) = \mathcal{L}^{-1}\Big(\{\phi \in (\mathbb{R}^n)^* : \phi(y-x) \leq 0 \ \ \forall \ y \in K\}\Big). 
\end{align*}
\end{lemma}
\begin{proof} Let $v \in \mathrm{NC}(K,x)$. Of course, there is nothing to prove for the case $v = 0$, and hence we may assume that $v$ is a non-zero vector. If $\mathbf{h}$ is the hyperplane such that $v \dashv_B \mathbf{h}$, then $\mathbf{h}$ supports $K$ at $x$. Denote $\phi = \mathcal{L}(v)$. If $y \in K$, then $y$ does not lie in the same half-space determined by $\mathbf{h}$ as $v$ (recall that $v$ is an outer normal), and hence we may write
\begin{align*} y = \alpha v + z,
\end{align*} 
for some $z \in \mathbf{h}$ and some $\alpha \leq 0$. Thus,
\begin{align*} \phi(y-x) = \mathcal{L}(v)\cdot(y-x) = \mathcal{L}(v)\cdot(\alpha v+z) = \alpha||v||^2 \leq 0,
\end{align*}
and this shows the inclusion ``$\subseteq$". Now assume that $\phi$ is a non-zero functional and observe that if $\phi(y-x) \leq 0$ for every $y \in K$, then $\mathbf{h}:=\mathrm{ker}(\phi)$ supports $K$ at $x$. Then, if $v$ is a unit and outward pointing vector such that $v \dashv_B \mathbf{h}$, it follows that $\phi = \mathcal{L}(\lambda v)$ for some $\lambda > 0$. This gives the reverse inclusion. 

\end{proof}

\begin{remark} The ``functional versions" of normal cones and sub-differentials can be used to investigate the case where the norms are not smooth or strictly convex. In this direction, we refer the reader to \cite{penot}.  
\end{remark}

\section{Differentiability of distance functions}\label{distdiff}

Throughout this section we will \textbf{always} assume that the involved norms are smooth and strictly convex. Let $K\subseteq(\mathbb{R}^n,||\cdot||)$ be a convex body which is not necessarily smooth or strictly convex. For a given $x \in \mathbb{R}^n\setminus K$ the metric projection $p_K(x) \in \partial K$ is unique, and hence the outer normal $\eta_K(x)$ defined as
\begin{align*} \eta_K(x) := \frac{x-p_K(x)}{||x-p_K(x)||}
\end{align*}
is unique. Denote by $d_K$ the distance function to $K$, defined as
\begin{align*} d_K(x) := \mathrm{dist}(x,K).
\end{align*}
In the next theorem we discuss the differentiability of $d_K$.

\begin{teo} The function $d_K$ is differentiable in $\mathbb{R}^n\setminus K$. Moreover, we have
\begin{align*} \nabla d_K(x) = \eta_K(x)
\end{align*}
for any $x \in \mathbb{R}^n\setminus K$.
\end{teo}
\begin{proof} First observe that, for each $c > 0$, the sub-level set $\{d_K \leq c\}$ is precisely the convex body $K + cB$. Assume that $d_K(x) = c$, and let $\mathbf{h}$ be the (unique) supporting hyperplane of $K+cB$ at $x$. Assume that $v \in \partial d_K(x)$. Due to Theorem \ref{normsubgradgeom}, the hyperplane which is right-orthogonal to $v$ supports $K+cB$ at $x$, and hence $v \dashv_B \mathbf{h}$. From Propositions \ref{outernormal} and \ref{birkhoffparallelset} we get that $v$ is a multiple of $\eta_K(x)$. It follows that $\partial d_K(x) \subseteq \mathrm{span}\{\eta_K(x)\}$. From Proposition \ref{propsun} we have that 
\begin{align*} d_K(x+t\eta_K(x)) = d_K(x) + t,
\end{align*}
for $t \in \mathbb{R}$ small enough. From that equality we get immediately that
\begin{align*} (d_K)'_+(x,\eta_K(x)) = 1 = -(d_K)'_+(x,-\eta_K(x)).
\end{align*}
Now let $\alpha\eta_K(x) \in \partial d_K(x)$. Lemma \ref{maxlemma} implies that
\begin{align*} 1 = (d_K)'_+(x,\eta_K(x)) \geq \mathcal{L}(\alpha\eta_K(x))\cdot\eta_K(x) = \alpha||\eta_K(x)||^2 = \alpha.
\end{align*}
Finally, we have
\begin{align*} -1 = (d_K)'_+(x,-\eta_K(x)) \geq \mathcal{L}(\alpha\eta_K(x))\cdot(-\eta_K(x)) = -\alpha||\eta_K(x)||^2 = -\alpha.
\end{align*}
It follows that $\alpha = 1$. Hence $\partial d_K(x) = \{\eta_K(x)\}$, that is, the norm sub-differential of $d_K$ at $x$ is a singleton. From Corollary \ref{singletondiff} we get that $d_K$ is differentiable at $x$, and $\nabla d_K(x) = \eta_K(x)$. 

\end{proof}

\begin{coro} The norm gradient $\nabla d_K$ of the distance function $d_K$ is continuous in $\mathbb{R}^n\setminus K$. In particular, $d_K$ is a function of class $C^1$. 
\end{coro}
\begin{proof} This follows immediately from the equality
\begin{align*} x = p_K(x) + \mathrm{dist}(x,K)\cdot\eta_K(x),
\end{align*}
which holds for any $x \in\mathbb{R}^n\setminus K$. Since the distance function $\mathrm{dist}(\cdot,K)$ and the metric projection $p_K$ are continuous functions in $\mathbb{R}^n\setminus K$ (see Proposition \ref{metricprojectioncontinuous}) we get that $\eta_K(x)$ is continuous in $\mathbb{R}^n\setminus K$. Consequently, $\nabla d_K$ is continuous in $\mathbb{R}^n\setminus K$, and from Remark \ref{continuousequiv} we get that $d_K$ is a function of class $C^1$ in $\mathbb{R}^n\setminus K$. 
 
\end{proof}

Observe that in the interior of $K$ we clearly have that $d_K$ is differentiable, and $\nabla d_K(x) = \{0\}$. In fact, one just has to recall that $d_K$ is constant in $K$. Next we investigate what happens on the boundary of $K$. We show that $d_K$ is not differentiable on $\partial K$, and characterize its norm sub-differential at these points. 

\begin{teo} For each $x \in \partial K$ we have
\begin{align*} \partial d_K(x) = \mathrm{NC}(K,x).
\end{align*}
In particular, $d_K$ is not differentiable at a boundary point of $K$.
\end{teo}
\begin{proof} The fact that $0 \in \partial d_K(x)$ follows from the property that $d_K(x) = 0$ is the global minimum of $d_K$. Now let $v$ be a (unit) outer normal of $K$ at $x$, and assume that $\mathbf{h}$ is the supporting hyperplane of $K$ at $x$ such that $v \dashv_B\mathbf{h}$. Since $d_K(x+tv) = t$ for $t \geq 0$, we have that $(d_K)'_+(x,v) = 1$. If $\theta \in (0,+\infty)$, then we will prove that $\theta v \in \partial d_K(x)$. Let $y \in \mathbb{R}^n$, and write $y = \lambda\theta v + z+ x$ for some $z \in \mathbf{h}$ and some $\lambda \in \mathbb{R}$. First, assume that $\lambda > 0$. From (\ref{ineqconvex1}) we have
\begin{align*} d_K(\lambda\theta v+z+x) - d_K(x) \geq (d_K)'_+(x,\lambda\theta v+z).
\end{align*}
Also, we claim that $(d_K)'_+(x,\lambda\theta v + z) \geq (d_K)'_+(x,\lambda\theta v)$. Indeed, observe that  $d_K(x+ tv) = t$ for any $t > 0$, and that Proposition \ref{birkhoffparallelset} implies that $z$ is a supporting direction of $\{d_K \leq t\}$ at $x+tv$. Hence $d_K(x+\alpha v+\beta z) \geq d_K(x+\alpha v)$ for any $\alpha > 0$ and $\beta \in \mathbb{R}$. Consequently,
\begin{align*} (d_K)'_+(x,\lambda\theta v + z) = \lim_{\varepsilon\rightarrow 0^+}\frac{d_K(x+\varepsilon(\lambda\theta v+z))-d_K(x)}{\varepsilon} \geq \lim_{\varepsilon\rightarrow 0^+}\frac{d_K(x+\varepsilon\lambda\theta v)-d_K(x)}{\varepsilon} = \\ = (d_K)'_+(x,\lambda\theta v).
\end{align*}
Since $(d_K)'_+(x,\lambda\theta v) = \lambda\theta(d_K)'_+(x,v) = \lambda\theta$, we get
\begin{align*} d_K(y) - d_K(x) \geq (d_K)'_+(x,\lambda\theta v+z) \geq \lambda\theta.
\end{align*} 
On the other hand, we have that
\begin{align*} \mathcal{L}(\theta v)\cdot(y-x) = \mathcal{L}(\theta v)\cdot(\lambda v+z) = \lambda\theta||v||^2 = \lambda\theta,
\end{align*}
and this concludes the case $\lambda > 0$. If $\lambda \leq 0$, then we recall again that $d_K(x)= 0$ is the global minimum of $d_K$, and write
\begin{align*} \mathcal{L}(\theta v)\cdot(y-x) = \mathcal{L}(\theta v)\cdot(\lambda v + z) = \lambda\theta||v||^2 = \lambda\theta \leq 0 \leq d_K(y) - d_K(x). 
\end{align*}
This shows that $\theta v \in \partial d_K(x)$. Hence we have the inclusion $\partial d_K(x)\supseteq\mathrm{NC}(K,x)$. Now assume that $v \in \partial d_K(x)$. For any $y \in K$ we have
\begin{align*} 0 = d_K(y)-d_K(x) \geq \mathcal{L}(v)\cdot(y-x).
\end{align*}
Consequently, we get from Lemma \ref{functionalcone} that $v \in \mathrm{NC}(K,x)$. This gives the remaining inclusion $\partial d_K(x) \subseteq \mathrm{NC}(K,x)$. 

\end{proof}

\end{document}